\documentclass[11pt]{amsart}
\usepackage{amsfonts}
\usepackage{amsmath}
\usepackage{amsthm}
\usepackage{bm}
\usepackage[dvips]{graphicx}
\usepackage{caption}
\usepackage{color}

\usepackage{bigints}

\numberwithin{equation}{section}

\allowdisplaybreaks

\theoremstyle{plain}
\newtheorem{theorem}{Theorem}[section]
\newtheorem{lemma}[theorem]{Lemma}
\newtheorem{proposition}[theorem]{Proposition}

\theoremstyle{definition}
\newtheorem{definition}[theorem]{Definition}
\newtheorem{remark}[theorem]{Remark}

\def\beqn{\begin{equation}}
\def\beqn*{$$}
\def\eeqn{\end{equation}}
\def\eeqn*{$$}

\newcommand{\BY}{{\bf Y}}
\newcommand{\BZ}{{\bf Z}}
\newcommand{\BV}{{\bf V}}

\newcommand{\bx}{{\bf x}}
\newcommand{\by}{{\bf y}}
\newcommand{\bz}{{\bf z}}

\def\P{\mathbb{P}}
\def\E{\mathbb{E}}
\def\Pn{\mathcal P_n}
\def\Pnp{\mathcal P_n^{\prime}}

\def\B{\mathcal B}
\def\ip{i^{\prime}}
\def\jp{j^{\prime}}
\def\hij{h^{(i,j)}}
\def\hijp{h^{(i,\jp)}}
\def\hipjp{h^{(\ip,\jp)}}
\def\hijipjp{h^{(i, j, i^{\prime}, j^{\prime})}}
\def\gij{g^{(i,j)}}
\def\gijp{g^{(i,j^{\prime})}}
\def\gipjp{g^{(i^{\prime},j^{\prime})}}

\def\Diip{D^{(i,\ip)}}
\def\muijj{\mu^{(i,j,j)}}
\def\muijjp{\mu^{(i,j,\jp)}}
\def\xiijipjp{\xi^{(i,j,\ip,\jp)}}
\def\mukt{\mu^{(k+2,1,1)}}
\def\Zij{Z^{(i,j)}}
\def\Zipjp{Z^{(\ip,\jp)}}

\newcommand{\reals}{{\mathbb R}}
\newcommand{\bbr}{\reals}
\newcommand{\bbn}{{\mathbb N}}

\newcommand{\X}{{\mathcal{X}}}
\newcommand{\Y}{{\mathcal{Y}}}
\newcommand{\Yp}{{\mathcal{Y}^{\prime}}}

\newcommand{\ellp}{{\ell^{\prime}}}

\newcommand{\one}{{\bf 1}}

\begin{document}

\bibliographystyle{abbrv}

\title[Sum of persistence barcodes]
{Limit Theorems for the Sum of Persistence Barcodes}
\author{Takashi Owada}
\address{Faculty of Electrical Engineering\\
Technion-Israel Institute of Technology \\
Haifa, 32000, Israel}
\email{takashiowada@ee.technion.ac.il}

\thanks{This research was supported by funding from the European Research Council under the European Union's
Seventh Framework Programme (FP/2007-2013) / ERC Grant Agreement n. 320422.}

\subjclass[2000]{Primary 60G70, 60F17. Secondary 60D05, 60G55, 55N35, 55U10.}
\keywords{Functional central limit theorem, Poisson limit theorem, random topology, persistent homology, Betti number. \vspace{.5ex}}

\begin{abstract}
Topological Data Analysis (TDA) refers to an approach that uses concepts from algebraic topology to study the ``shapes" of datasets.
The main 
focus of this paper is persistent homology, 
a ubiquitous tool in TDA. 
Basing our study on this, we investigate the topological dynamics of extreme sample clouds generated by a heavy tail distribution on $\bbr^d$. In particular, we establish various limit theorems for the sum of bar lengths in the persistence barcode plot, a graphical descriptor of persistent homology. It then turns out that the growth rate of the sum of the bar lengths and the properties of the limiting processes all depend on the distance of the region of interest in $\bbr^d$ from the weak core, that is, the area in which random points are placed sufficiently densely to connect with one another. If the region of interest becomes sufficiently close to the weak core, the limiting process involves a new class of Gaussian processes.
\end{abstract}

\maketitle

\section{Introduction} \label{sec:intro}

The aim of this study is to investigate the algebraic topological properties of heavy tail distributions, relying on a ubiquitous tool in Topological Data Analysis (TDA). Topological Data Analysis is a growing research area that broadly refers to the analysis of high-dimensional and incomplete datasets, using concepts from algebraic topology, while borrowing ideas and techniques from other fields in mathematics \cite{carlsson:2009}. The most typical approach to TDA is probably \textit{persistent homology}, which originated in computational topology and appears in a wide range of applications, including sensor networks \cite{desilva:ghrist:2007}, bioinformatics \cite{dabaghian:memoli:frank:carlsson:2012}, computational chemistry \cite{martin:thompson:coutsias:watson:2010}, manifold learning \cite{niyogi:smale:weinberger:2008}, and linguistics \cite{port:gheorghita:guth:clark:liang:dasu:marcolli:2015}.

A standard approach in TDA usually starts with a point cloud $\mathcal X = \{ x_1,\dots,x_n \}$ of points in $\bbr^d$, from which more complex sets are constructed. Two such examples are the union of balls $\bigcup_{i=1}^n B(x_i; t)$, where $B(x; t)$ is a closed ball of radius $t$ about the point $x$, and the  \textit{\v{C}ech complex},
$\check{C}(\mathcal{X},t)$.
\begin{definition}
\label{cech:defn}
Let $\mathcal{X}$ be a collection of points in $\bbr^d$ and $t$ be a positive number. Then, the \v{C}ech complex $\check{C}(\mathcal{X},t)$ is defined as follows.
\begin{enumerate}
\item The $0$-simplices are the points in $\mathcal{X}$.
\item A $p$-simplex $\sigma=[x_{i_0}, \dots, x_{i_p}]$ belongs to $\check{C}(\mathcal{X},t)$ whenever a family of closed balls $\bigl\{ B(x_{i_j}; t/2), \, j=0,\dots,p \bigr\}$ has a nonempty intersection.
\end{enumerate}
\end{definition}
\begin{figure}[!t]
\begin{center}
\includegraphics[width=9cm]{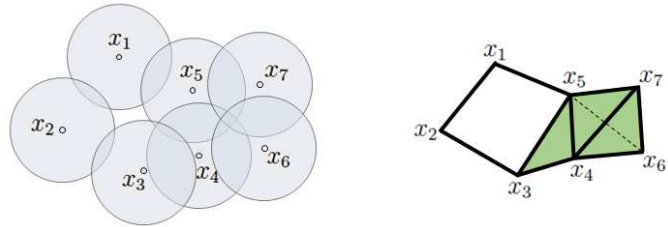}
\caption{{\footnotesize Take $\mathcal X = \{ x_1,\dots,x_7 \} \subset \bbr^2$. Since three balls with radius $t/2$ centered at $x_3, x_4, x_5$ have a common intersection, the $2$-simplex $[x_3,x_4,x_5]$ belongs to $\check{C}(\mathcal X; t)$. There also exists a $3$-simplex $[x_4,x_5,x_6,x_7]$, which adds a tetrahedron on the right figure. }}
\label{f:cech}
\end{center}
\end{figure}

In addition to the \v{C}ech complex, there are many other \textit{simplicial complexes}, such as the Vietoris-Rips and alpha complexes (see, e.g., \cite{ghrist:2014}). However, throughout the current paper, we concentrate on the \v{C}ech complex. One reason for doing so is its topological equivalence to the union of balls. Indeed, according to the Nerve theorem \cite{borsuk:1948}, the \v{C}ech complex and the union of balls are homotopy equivalent, and thus, they represent the same topological object. Furthermore, \v{C}ech complexes are regarded as higher-dimensional analogues of \textit{geometric graphs}, and therefore, many of the techniques developed thus far in random geometric graph theory (see, e.g., \cite{penrose:2003}) are also applicable to random \v{C}ech complexes.

A standard topological argument classifies objects such as \v{C}ech complexes, usually in terms of homological concepts, etc. Given a topological space $X$, the \textit{$0$-th homology group} $H_0(X)$ consists of elements that represent connected components in $X$, while for $k \geq 1$, the \textit{$k$-th homology group} $H_k(X)$ is generated by elements representing $k$-dimensional ``holes" or ``cycles" in $X$. Then, for $k \geq 0$, the \textit{$k$-th Betti number} $\beta_k(X)$ is defined as the rank of $H_k(X)$ and is the quantifier of topology that is central
 to the entire study in this paper. 
 More intuitively, $\beta_0(X)$ counts the number of connected components in $X$, while $\beta_k(X)$, $k \geq 1$, measures the number of $k$-dimensional holes or cycles in $X$. For example, a one-dimensional sphere, i.e.,\ a circle, has $\beta_0=1$, $\beta_1=1$,  and $\beta_k=0$ for all $k\geq 2$. A two-dimensional sphere has $\beta_0=1$, $\beta_1=0$, and $\beta_2=1$,  and all others  zero. In the case of a two-dimensional torus, the non-zero Betti numbers are  $\beta_0=1$, $\beta_1=2$, and $\beta_2=1$. At a more formal level, we need a rigorous coverage of homology theory (see, e.g., \cite{hatcher:2002} or \cite{vick:1994}); however, the essence of this paper can be captured without knowledge of homology theory. In the sequel, simply viewing $\beta_k(X)$ as the number of $k$-dimensional holes will suffice.
\begin{figure}[!t]
\begin{center}
\includegraphics[width=15cm]{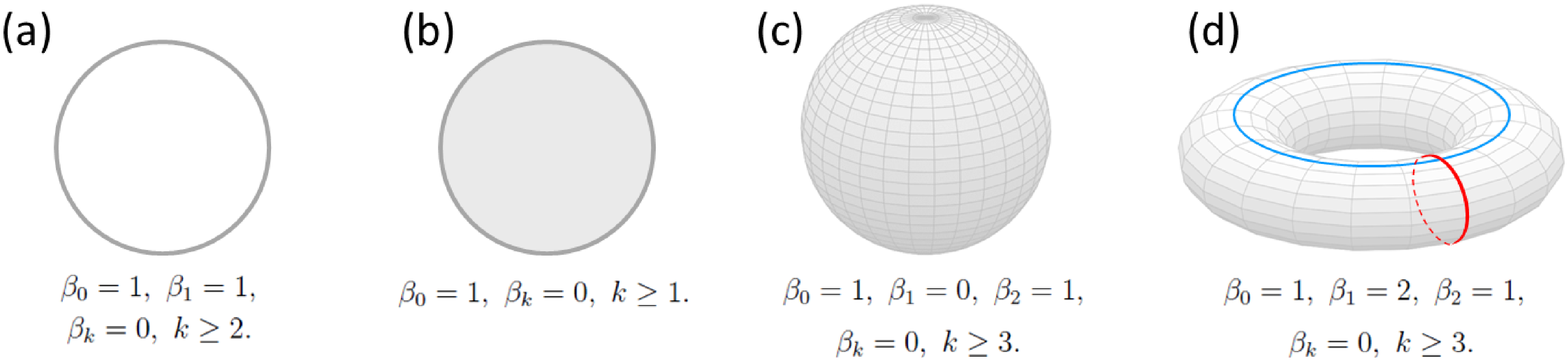}
\caption{{\footnotesize (a) One-dimensional sphere. (b) One-dimensional disk. (c) Two-dimensional sphere. (d) Two-dimensional torus. The Betti number $\beta_1$ of a two-dimensional sphere is zero; even if one winds a closed loop around the sphere, the loop ultimately vanishes as it moves upward (or downward) along the sphere until the pole. The Betti number $\beta_1$ of a two-dimensional torus is $2$ because of two independent closed loops (one is red and the other is blue). }}
\end{center}
\end{figure}

 Persistent homology keeps track of how topological features dynamically evolve in a filtered topological space. We do not give a formal description of persistent homology, but, alternatively, we present an illustrative example, which helps capture its essence. Readers interested in a more rigorous description of persistent homology may refer to \cite{edelsbrunner:letscher:zomorodian:2002}, \cite{zomorodian:carlsson:2005} and \cite{edelsbrunner:harer:2010}, while \cite{adler:bobrowski:borman:subag:weinberger:2010} and \cite{ghrist:2008} provide an elegant review of the topics in an accessible way for non-topologists. Let $\mathcal X_n = \{ X_1,\dots,X_n \}$ be a set of random points on $\bbr^d$, drawn from an unknown manifold $\mathcal M \subset \bbr^d$. First, we construct a union of balls
$$
U(t) := \bigcup_{i=1}^n B(X_i; t), \ \ \ t\geq0,
$$
which defines a random filtration generated by balls with increasing radii $t\to \infty$, that is, $U(s) \subset U(t)$ holds for all $0 \leq s \leq t$. By virtue of the Nerve theorem, this filtration conveys the same homological information as a collection of \v{C}ech complexes $\bigl\{  \check{C} (\mathcal X_n; t), \,  t\geq0 \bigr\}$. Utilizing $\bigl\{ U(t), \, t\geq0  \bigr\}$ or $\bigl\{  \check{C} (\mathcal X_n; t), \,  t\geq0 \bigr\}$, we wish to recover the homology of $\mathcal M$. We expect that, provided that $t$ is suitably chosen, the union of balls $U(t)$ is homotopy equivalent to $\mathcal M$ and hence its homology is the same as $\mathcal M$. In general, however, selecting such an appropriate $t$ is not easy at all. To make this more transparent, we consider an example for which $\mathcal M$ represents an annulus (Figure \ref{f:persistence}). In this case, if $t$ is chosen to be too small, $U(t)$ is homotopy equivalent to many distinct points, implying that we fail to recover the homology of an annulus. On the other hand, if $t$ is extremely large, then $U(t)$ becomes contractible (i.e., can deform into a single point continuously) and, once again, $U(t)$ does not recover the homology of an annulus.
\begin{figure}[!t]
\begin{center}
\includegraphics[width=11.5cm]{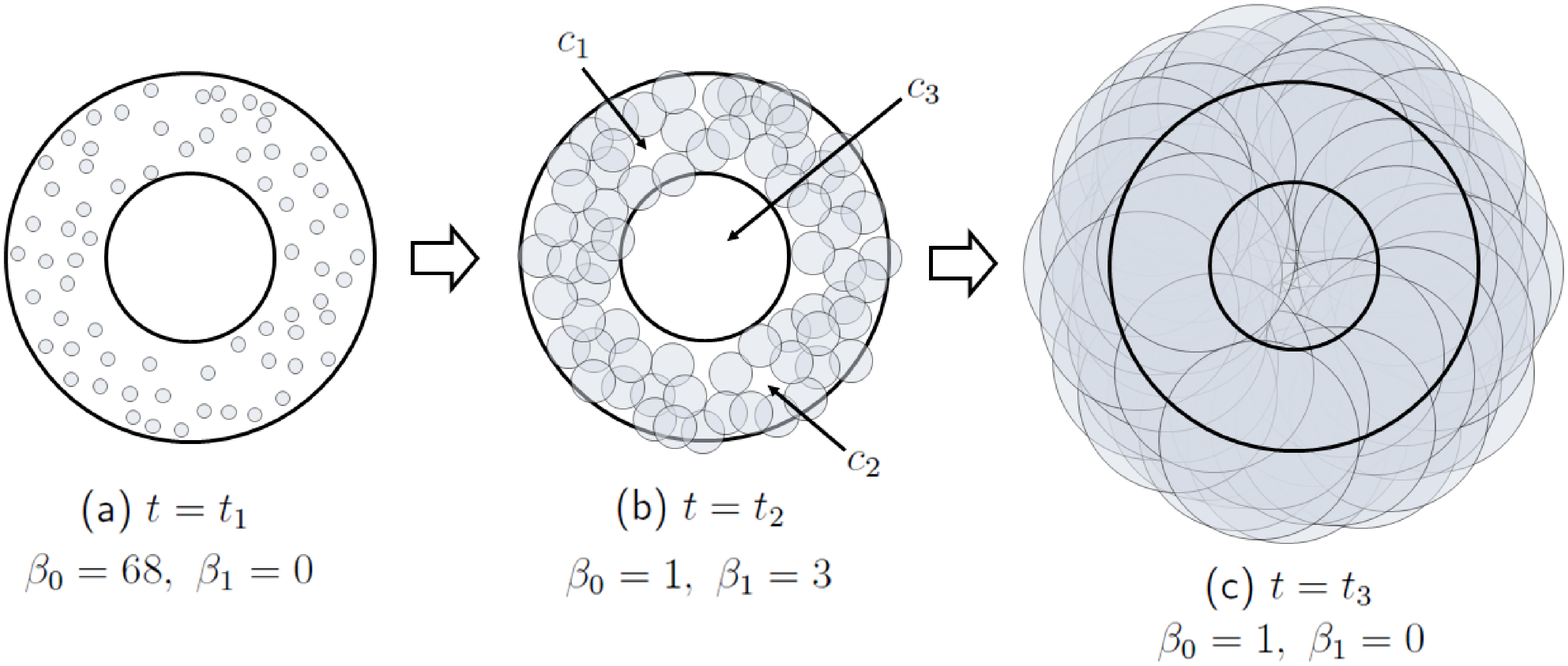}
\caption{{\footnotesize Many random points are scattered over an annulus. We increase the radius $t$ of the balls about these random points. }}
\label{f:persistence}
\end{center}
\end{figure}
\begin{figure}[!t]
\begin{center}
\includegraphics[width=12cm]{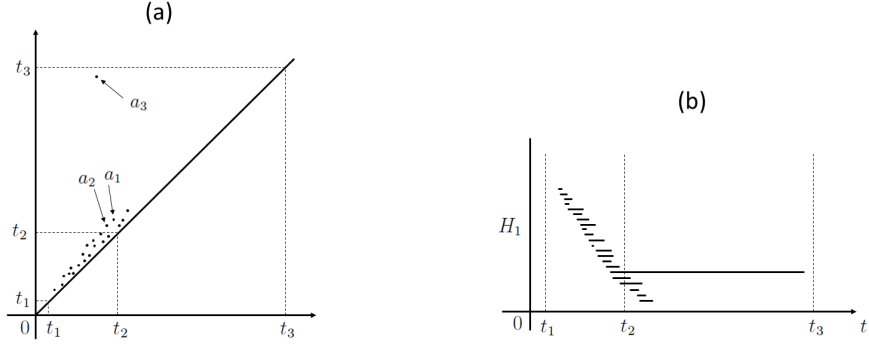}
\caption{{\footnotesize (a) Persistence diagram for the first homology group $H_1$ represented by one-dimensional holes. In Figure \ref{f:persistence}, there exist two small holes $c_1$ and $c_2$ when $t=t_2$. The lifetimes of these holes are so short that they are represented by the points $a_1$ and $a_2$ near the diagonal line. On the other hand, $c_3$ is a robust hole, and thus, the corresponding point $a_3$ is placed far from the diagonal line. (b) Persistence barcode plot for $H_1$. The vertical line at level $t_2$ intersects horizontal bars three times, meaning that there are three holes when $t=t_2$. Although two of these quickly vanish, the remaining one has the largest persistence and generates the longest bar. }}
\label{f:diagram-plot}
\end{center}
\end{figure}

 Persistent homology can extract the robust homological information of $\mathcal M$ by treating a possible range of $t$ simultaneously. Typically, persistent homology can be visualized by two equivalent graphical descriptors known as the \textit{persistence diagram} and \textit{persistence barcode plot}. The persistence diagram consists of a multiset of points in the plane $\bigl\{ (b_i, d_i): i=1,\dots, m,\ 0 \leq b_i < d_i \leq \infty \bigr\}$, where each pair $(b_i, d_i)$ describes the birth time and death time of each hole (or connected component). Alternatively, if we represent the pair $(b_i, d_i)$ as an interval $[b_i, d_i]$, we obtain a set of horizontal bars, called the persistence barcode plot.

For the annulus example in Figure \ref{f:persistence}, as we increase the radius $t$, many small one-dimensional holes appear and quickly disappear (e.g., the holes $c_1$ and $c_2$). Since the birth time and death time of these non-robust holes are close to each other, they are expressed in the persistence diagram as the points near the diagonal line (see the points $a_1$ and $a_2$ in Figure \ref{f:diagram-plot} (a)). The points near the diagonal line are usually viewed as ``topological noise." In contrast, a robust hole for the annulus denoted by $c_3$ in Figure \ref{f:persistence} has a much longer lifetime than any other small hole, and therefore, it can be represented by the point $a_3$ placed far above the diagonal line. From the viewpoint of the persistence barcode plot in Figure \ref{f:diagram-plot} (b), the hole $c_3$ generates the longest bar, whereas other small holes generate only much shorter bars.

Given a set of intervals $[b_i, d_i]$, $i=1,\dots,m$, in the persistence barcode plot for the $k$-th homology group $H_k$ (for short, we call it $k$-th persistence barcode plot), the quantity we explore in the present paper is the \textit{lifetime sum} up to parameter $t$ defined by
\begin{equation}  \label{e:def.lifetime.sum}
L_k(t) := \sum_{i=1}^m \bigl( d_i(t) - b_i(t) \bigr),
\end{equation}
where
$$
b_i(t) = \begin{cases}  b_i & \text{if }  b_i \leq t,  \\
t & \text{if } b_i > t,
\end{cases}
\qquad \qquad
d_i(t) = \begin{cases}  d_i & \text{if }  d_i \leq t,  \\
t & \text{if } d_i > t.
\end{cases}
$$
\begin{figure}[!t]
\begin{center}
\includegraphics[width=8cm]{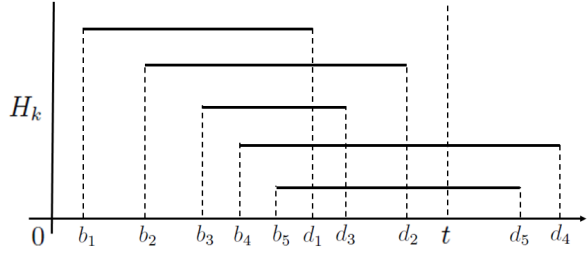}
\caption{{\footnotesize Persistence barcode plot for $H_k$. A set of horizontal bars represents the birth and death of $k$-dimensional holes. In this case, the lifetime sum up to parameter $t$ is given by $L_{k,n}(t) = \sum_{i=1}^3 (d_i - b_i) + (t-b_4) + (t - b_5)$. }}
\label{f:barcode1}
\end{center}
\end{figure}

Utilizing these topological tools developed in TDA, we investigate the topological dynamics of \textit{extreme sample clouds} lying far away from the origin, which are generated by heavy tail distributions on $\bbr^d$. The study of the geometric and topological properties of extreme sample clouds in a high-dimensional space belongs to Extreme Value Theory (EVT). Indeed, over the last decade or so, many studies have provided geometric descriptions of multivariate extremes in view of point process theory, among them \cite{balkema:embrechts:2007}, \cite{balkema:embrechts:nolde:2010}, and \cite{balkema:embrechts:nolde:2013}.  In particular, Poisson limits of point processes with a U-statistic structure were discussed in \cite{dabrowski:dehling:mikosch:sharipov:2002} and \cite{schulte:thale:2012}, the latter also including a number of stochastic geometry examples. Furthermore, in \cite{owada:adler:2016} a recent extensive study of the general point process convergence of extreme sample clouds, leading to limit theorems for Betti numbers of extremes, is reported. The main contribution in \cite{owada:adler:2016} is a probabilistic investigation into a
 layered structure consisting of a collection of ``rings" around the origin, with each ring containing extreme random points that exhibit different topological behaviors in terms of the Betti numbers. More formally, this ring-like structure is referred to as \textit{topological crackle}, which was originally reported in \cite{adler:bobrowski:weinberger:2014}. We remark also that there has been increasing interest in the limiting behaviors of random simplicial complexes, which are not necessarily related to extremes; see \cite{kahle:2011}, \cite{kahle:meckes:2013}, \cite{yogeshwaran:adler:2015},    \cite{yogeshwaran:subag:adler:2014}, and \cite{bobrowski:mukherjee:2015}. These papers derive various limit theorems for the Betti numbers of the random \v{C}ech complexes $\check{C}(\mathcal X_n; r_n)$, with $\mathcal X_n$ a random point set in $\bbr^d$ and $r_n$ a threshold radius decreasing to $0$.

The organization of this paper is as follows. First we provide a formal setup of our extreme sample clouds and express the lifetime sum of extremes as a simple functional of the corresponding Betti numbers. We observe that the nature of limit theorems for the lifetime sum of extremes depends crucially on the distance of the region of interest from the origin. The asymptotics of the lifetime sum exhibits completely different topological features according to the region examined.
The persistent homology originated in algebraic topology, and thus, there are only a limited number of probabilistic and statistical studies that have treated it. The present paper contains some of the earliest and
most comprehensive results obtained by examining persistent homology from a pure probabilistic viewpoint, two other papers being \cite{hiraoka:shirai:2015} and \cite{bobrowski:kahle:skraba:2015}. The interdisciplinary studies between statistics and persistent homology include, for example, \cite{fasy:lecci:rinaldo:wasserman:balakrishnan:singh:2014}, \cite{bubenik:2015}, and \cite{kusano:fukumizu:hiraoka:2016}. 

Before commencing the main body of the paper, we remark that all the random points in this
paper are assumed to be generated by an inhomogeneous Poisson point process on $\bbr^d$ with intensity
$nf$. In our opinion, all the limit theorems derived in this paper can be carried over to a usual iid random
sample setup by a standard ``de-Poissonization" argument; see Section 2.5 in \cite{penrose:2003}. This is, however, a little more technical and challenging, and therefore, we decided to concentrate on the simpler setup of an inhomogeneous Poisson point process. Furthermore, we consider only spherically symmetric
distributions. Although the spherical symmetry assumption is far from being crucial, we adopt it
to avoid unnecessary technicalities.

\section{Limit Theorems for the Sum of Bar Lengths}

Let $(X_i, \, i \geq 1)$ be an iid sequence of $\bbr^d$-valued random variables with common spherically symmetric density $f$ of a regularly varying tail. Let $S_{d-1}$ be the $(d-1)$-dimensional unit sphere in $\bbr^d$. Assume that for any $\theta \in S_{d-1}$ (equivalently for some $\theta \in S_{d-1}$) and for some $\alpha > d$,
\begin{equation}  \label{e:RV.tail}
\lim_{r\to \infty} f(rt\theta)/f(r\theta) = t^{-\alpha} \ \ \text{for every } t>0.
\end{equation}
Denoting by $RV_{\gamma}$ a family of regularly varying functions (at infinity) with exponent $\gamma \in \bbr$, this can be written as $f \in RV_{-\alpha}$.
Let $N_n$ be a Poisson random variable with mean $n$, independent of $(X_i)$, and $\Pn = \{ X_1,\dots,X_{N_n} \}$ denote an inhomogeneous Poisson point process on $\bbr^d$ with intensity $nf$.

Given a sequence $(R_n, \, n \geq 1)$ growing to infinity and a non-negative number $t \geq 0$, we denote by $\check{C}\bigl( \Pn \cap B(0;R_n)^c; t \bigr)$ a \v{C}ech complex built over random points in $\Pn$ lying outside a growing ball $B(0;R_n)$. Then, a family of \v{C}ech complexes
\begin{equation}  \label{e:filtration.cech}
\Bigl\{  \check{C} \bigl( \Pn \cap B(0; R_n)^c; t \bigr), \, t \geq 0  \Bigr\}
\end{equation}
constitutes a ``random filtration" parametrized by $t\geq 0$. That is, we have for all $0 \leq s \leq t$,
$$
\check{C} \bigl( \Pn \cap B(0; R_n)^c; s \bigr) \subset \check{C} \bigl( \Pn \cap B(0; R_n)^c; t \bigr).
$$

Choosing a positive integer $k \geq 1$, which remains fixed hereafter, we denote the $k$-th Betti number of the \v{C}ech complex by
$$   
\beta_{k,n}(t) := \beta_k \Bigl( \check{C}\bigl( \Pn \cap B(0;R_n)^c; t \bigr) \Bigr)
= \beta_k \biggl(\ \bigcup_{X \in \Pn \cap B(0; R_n)^c} B(X; t) \biggr),
$$
where the second equality is justified by homotopy equivalence between the \v{C}ech complex and the union of balls.
\begin{figure}[!t]
\begin{center}
\includegraphics[width=9cm]{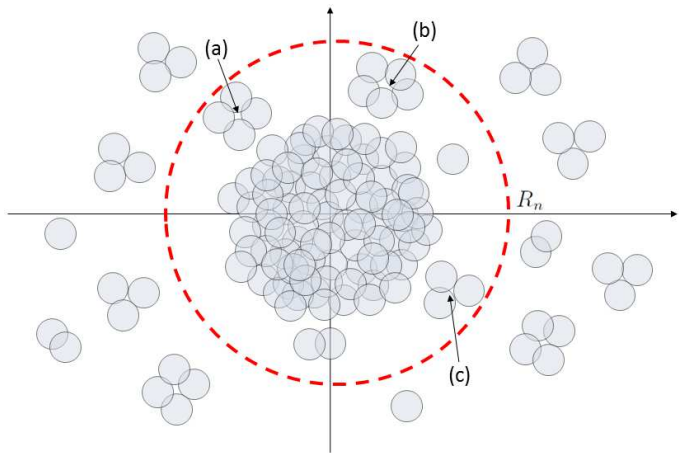}
\caption{{\footnotesize For $k=1$, $d=2$. The Betti number $\beta_{1,n}(t)$ counts one-dimensional holes outside $B(0; R_n)$, while ignoring holes inside the ball (e.g., (a), (b), and (c)). }}
\label{f:betti-cech}
\end{center}
\end{figure}

Here, we provide a key relation between the $k$-th Betti number and the lifetime sum of the $k$-th persistent homology associated with the filtration \eqref{e:filtration.cech}. Denote by $L_{k,n}(t)$ the lifetime sum in the $k$-th persistence barcode plot up to parameter $t$, as constructed in \eqref{e:def.lifetime.sum}.
 Then, it holds that
\begin{equation}  \label{e:sum.barcode}
L_{k,n}(t) = \int_0^t \beta_{k,n}(s) ds, \ \ t\geq0.
\end{equation}
The proof of \eqref{e:sum.barcode} is elementary. In the persistence barcode plot, the Betti number $\beta_{k,n}(s)$ represents the number of times the vertical line at level $s$ intersects the horizontal bars (Figure \ref{f:barcode2}). Therefore, the integration of $\beta_{k,n}(s)$ from $0$ to $t$ equals the sum of the bar lengths $L_{k,n}(t)$. For more formal proof, one may refer to Proposition 2.2 in \cite{hiraoka:shirai:2015}. Clearly, \eqref{e:sum.barcode} may be viewed as generating a stochastic process in the parameter $t\geq0$ with continuous sample paths, and its limiting properties are central to this paper, for which we derive various limit theorems in the sequel.
\begin{figure}[!t]
\begin{center}
\includegraphics[width=8cm]{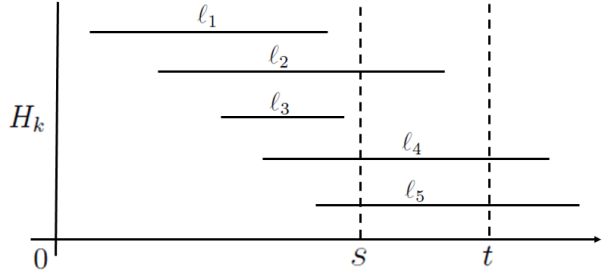}
\caption{{\footnotesize $k$-th persistence barcode plot. The lifetime sum up to parameter $t$ is $L_{k,n}(t) = \sum_{i=1}^5\ell_i$. The vertical line at level $s$ intersects the horizontal bars three times, implying that $\beta_{k,n}(s) = 3$. The integration of $\beta_{k,n}(s)$ from $0$ to $t$ coincides with $\sum_{i=1}^5\ell_i$. }}
\label{f:barcode2}
\end{center}
\end{figure}

The behavior of \eqref{e:sum.barcode} splits into three different regimes, each of which is characterized by the growth rate of $R_n$:
\begin{align*}
&(i) \ n^{k+2}R_n^d f(R_ne_1)^{k+2} \to 1, \ \ \ n\to\infty, \\
&(ii) \  n^{k+2}R_n^d f(R_ne_1)^{k+2} \to \infty,   \ \ nf(R_ne_1) \to 0, \ \ \ n\to\infty,  \\
&(iii) \ nf(R_ne_1) \to \lambda, \ \  \ n\to\infty, \text{ for some } \lambda \in (0,\infty)
\end{align*}
with $e_1 = (1,0,\dots,0) \in \bbr^d$.
Since $(R_n)$ in case $(i)$ grows fastest, the occurrence of $k$-dimensional holes outside $B(0; R_n)$ is the least likely of the three regimes. In contrast, the $R_n$ determined by $(iii)$ grows most slowly, which implies that the occurrence of $k$-dimensional holes outside $B(0; R_n)$ is the most likely of the three regimes. In the following, we establish the limit theorems for $L_{k,n}(t)$ in all three regimes.

Before proceeding to specific subsections, we need to introduce one important notion.
\begin{definition}  \label{def.weak.core}
Let $f$ be a spherically symmetric density on $\bbr^d$. A \textit{weak core} is a centered ball
$B(0; R_n^{(w)})$ such that $nf(R_n^{(w)}e_1) \to 1$ as $n \to
\infty$.
\end{definition}

Weak cores are balls, centered at the origin with growing radii as $n$ increases, in which random points are placed so densely that the balls with fixed (e.g, unit) radius about these random points become highly connected with one another and form a giant component of a geometric graph. For example, if $f$ has a power-law tail
$$
f(x) =  C/\bigl( 1+\| x \|^\alpha \bigr), \ \ x \in \bbr^d
$$
for some $\alpha > d$ and normalizing constant $C$ ($\| \cdot \|$ denotes a Euclidean norm), then the radius of a weak core is given by $R_n^{(w)} = (Cn)^{1/\alpha}$. The properties of a weak core, together with those of the related notion of a \textit{core}, were carefully explored in \cite{owada:2016} for a wide class of distributions. See also \cite{owada:adler:2016} and \cite{adler:bobrowski:weinberger:2014}. Note that the $R_n$ determined in $(iii)$ coincides with the radius of a weak core (up to multiplicative factors). Since there are essentially no holes inside the weak core, the case in which $(R_n)$ satisfies $nf(R_ne_1) \to \infty$, $n \to \infty$ is expected to lead to the same asymptotic result as that in regime $(iii)$. Therefore, all non-trivial results regarding asymptotics of $L_{k,n}(t)$ can be completely covered by regimes $(i) - (iii)$.

\subsection{Poissonian Limit Theorem in the First Regime}  \label{s:first.regime}

First, we assume that $(R_n)$ satisfies condition $(i)$, i.e.,
\begin{equation}  \label{e:Rn.regime1}
n^{k+2}R_n^d f(R_ne_1)^{k+2} \to 1, \ \ \ n\to\infty.
\end{equation}
It is then elementary to check that $(R_n)$ is a regularly varying sequence (at infinity) with exponent
$$
R_{n} \in RV_{1/\bigl( \alpha - d/(k+2) \bigr)}.
$$
Since this exponent depends on $k$, we write $R_n = R_{k,n}$ whenever it becomes an asymptotic solution to \eqref{e:Rn.regime1}. Then, the resulting \v{C}ech complex lying outside  $B(0; R_n)$ is so sparse that there appear at most finitely many $k$-dimensional holes outside $B(0; R_n)$. Hence, the occurrence of $k$-dimensional holes outside $B(0; R_n)$ is seen to be ``rare," and, consequently, the limiting process for $L_{k,n}(t)$ is expressed as a natural functional of a certain Poisson random measure.

To define the limiting process more rigorously, we need some preparation. Let
\begin{equation}  \label{e:def.ht}
h_t(x_1,\dots,x_{k+2}) := \one \Bigl\{ \, \beta_k\bigl( \check{C}(x_1,\dots,x_{k+2}; t) \bigr) = 1\, \Bigr\}, \ \ x_i \in \bbr^d.
\end{equation}
This indicator function can be expressed as the difference between two other indicators:
\begin{align}
h_t(x_1,\dots,x_{k+2}) &= \one \Bigl\{ \bigcap_{j=1, \, j \neq j_0}^{k+2} B(x_j; t) \neq \emptyset \ \text{for all } j_0 \in \{ 1,\dots,k+2 \} \Bigr\} \label{e:decomp.ind} \\
&\quad - \one \Bigl\{ \bigcap_{j=1}^{k+2} B(x_j; t) \neq \emptyset \Bigr\} \notag \\
&:= h_t^+(x_1,\dots,x_{k+2}) - h_t^-(x_1,\dots,x_{k+2}).  \notag
\end{align}
This decomposition comes from the fact that $h_t(x_1,\dots,x_{k+2}) = 1$ if and only if $\{x_1,\dots,x_{k+2}  \}$ forms an \textit{empty $(k+1)$-simplex} with respect to $t$, i.e., for each $j_0 \in \{ 1,\dots,k+2 \}$, the intersection $\bigcap_{j=1, \, j \neq j_0}^{k+2} B(x_j; t)$ is non-empty, while $\bigcap_{j=1}^{k+2} B(x_j; t)$ is empty. Note that $h_t^+$ and $h_t^-$ are non-decreasing functions in $t$:
\begin{equation}  \label{e:increase.ht}
h_s^{\pm}(x_1,\dots,x_{k+2}) \leq h_t^{\pm}(x_1,\dots,x_{k+2})
\end{equation}
for all $x_1,\dots,x_{k+2} \in \bbr^d$ and $0\leq s \leq t$.
Hereafter, we denote $h(x_1,\dots,x_{k+2}) := h_1(x_1,\dots,x_{k+2})$ and $h^{\pm}(x_1,\dots,x_{k+2}) := h_1^{\pm}(x_1,\dots,x_{k+2})$.

Next, we give a Poissonian structure to the limiting process. Let
\begin{equation}  \label{e:def.Ck}
C_k = \frac{s_{d-1}}{(k+2)! \bigl( \alpha (k+2) - d \bigr)}\,,
\end{equation}
where $s_{d-1}$ is a surface area of the $(d-1)$-dimensional unit sphere in $\bbr^d$. Writing $\lambda_k$ for the Lebesgue measure on $(\bbr^d)^{k+1}$, the \textit{Poisson random measure}  $M_k$ with intensity measure $C_k \lambda_k$ is defined by the finite-dimensional distributions
$$
\P \bigl\{ M_k(A)=m \bigr\} = e^{-C_k\lambda_k(A)} \bigl( C_k\lambda_k(A) \bigr)^m/m!, \ \ \ m=0,1,2,\dots
$$
for all measurable $A \subset (\bbr^d)^{k+1}$ with $\lambda_k(A)<\infty$. Furthermore, if $A_1,\dots,A_m$ are disjoint subsets in $(\bbr^d)^{k+1}$, then $M_k(A_1), \dots, M_k(A_m)$ are independent.

We now state the main result of this subsection, the proof of which is, however, deferred to the Appendix. In the following, $\Rightarrow$ denotes weak convergence. All weak convergences hereafter are basically either in the space $D[0,\infty)$ of right-continuous functions with left limits or in the space $C[0,\infty)$ of continuous functions.
\begin{theorem}  \label{t:sparse.poisson}
Suppose that $R_n = R_{k,n}$ satisfies \eqref{e:Rn.regime1}. Then,
\begin{equation}  \label{e:limit.betti.first}
\beta_{k,n}(t) \Rightarrow V_k(t):= \int_{(\bbr^d)^{k+1}} h_t(0,\by) M_k(d\by) \ \ \text{in } D[0,\infty).
\end{equation}
Furthermore,
\begin{equation}  \label{e:limit.barcode.first}
L_{k,n}(t) \Rightarrow \int_0^tV_k(s)ds \ \ \text{in } C[0,\infty).
\end{equation}
\end{theorem}

Recalling the definition of $h_t$, one may state that the $k$-dimensional holes contributing to the limit are always formed by connected components on $k+2$ vertices, while other components on more than $k+2$ vertices never appear in the limit. Since there need to be at least $k+2$ vertices to form a single $k$-dimensional hole, all the $k$-dimensional holes  remaining in the limit are necessarily formed by components of the smallest size.

Because of the decomposition \eqref{e:decomp.ind}, we can denote $\BV_k = \bigl( V_k(t), \, t\geq 0 \bigr)$ as
\begin{align*}
V_k(t) &= \int_{(\bbr^d)^{k+1}} h_t^+(0,\by) M_k(d\by) - \int_{(\bbr^d)^{k+1}} h_t^-(0,\by) M_k(d\by) \\
&:= V_k^+(t) - V_k^-(t).
\end{align*}
The following proposition shows that $\BV_k^+$ and $\BV_k^-$ can be represented as a time-changed Poisson process.
\begin{proposition}  \label{p:diff.poisson}
The process $\BV_k^{\pm}$ is represented in law as
$$
\bigl( V_k^{\pm}(t), \, t\geq 0 \bigr) \stackrel{d}{=} \Bigl( N_k^{\pm}\bigl( t^{d(k+1)} \bigr), \, t\geq 0 \Bigr),
$$
where $N_k^{\pm}$ is a Poisson process with intensity $C_k\int_{(\bbr^d)^{k+1}} h^{\pm}(0,\by) d\by$.
\end{proposition}
\begin{proof}
It is straightforward to calculate the moment generating function of $(V_k^{\pm}(t_1), \dots, V_k^{\pm}(t_m))$ for $0 \leq t_1 < \dots < t_m <\infty$. For $\lambda_1, \dots, \lambda_m \geq 0$, we have
\begin{equation}  \label{e:mgf.Vk}
\E \Bigl\{ \exp \bigl\{ -\sum_{j=1}^m \lambda_j V_k^{\pm}(t_j) \bigr\} \Bigr\} = \exp \Bigl\{ -C_k \int_{(\bbr^d)^{k+1}} \bigl( 1-e^{-\sum_{j=1}^m \lambda_j h_{t_j}^{\pm}(0,\by)} \bigr)d\by \Bigr\}.
\end{equation}
Exploiting this result, one can easily see that $\BV_k^{\pm}$ has independent increments, while for $0\leq s \leq t$, $V_k^{\pm}(t) - V_k^{\pm}(s)$ has a Poisson law with mean $C_k\int_{(\bbr^d)^{k+1}} h^{\pm}(0,\by) d\by\, (t^{d(k+1)}-s^{d(k+1)})$. Now, the claim follows.
\end{proof}
\begin{remark}
By the moment generating function \eqref{e:mgf.Vk}, it is easy to see that for each $t\geq0$, $V_k(t)$ has a Poisson distribution with mean $C_k \int_{(\bbr^d)^{k+1}} h(0,\by)\, d\by\, t^{d(k+1)}$. Nevertheless, the process $\BV_k$ cannot be represented as a (time-changed) Poisson process, since the sample paths of $\BV_k$ allow for both upward and downward jumps.
\end{remark}

\subsection{Functional Central Limit Theorem in the Second Regime}
 \label{s:second.regime}

In this subsection, we turn to the second regime, which is characterized by
\begin{equation} \label{e:second.Rn}
n^{k+2} R_n^d f(R_ne_1)^{k+2} \to \infty, \ \  nf(R_ne_1) \to 0, \ \ \ n \to \infty,
\end{equation}
for which $(R_n)$ exhibits a slower divergence rate than that in the previous regime. Thus, we expect that, in an asymptotic sense, there  appear infinitely many $k$-dimensional holes outside $B(0;R_n)$, and accordingly, instead of a Poissonian limit theorem, some sort of functional central limit theorem (FCLT) governs the behavior of $L_{k,n}(t)$.

To formulate the limiting process for $L_{k,n}(t)$, we need some preliminary work. As before, let $\lambda_k$ denote the Lebesgue measure on $(\bbr^d)^{k+1}$ and $C_k$ a positive constant given in \eqref{e:def.Ck}. Denote by $G_k$ a \textit{Gaussian $C_k\lambda_k$-noise}, such that
$$
G_k(A) \sim \mathcal N \bigl( 0, C_k\lambda_k(A) \bigr)
$$
for measurable sets $A \subset (\bbr^d)^{k+1}$ with $\lambda_k(A) < \infty$, and if $A \cap B = \emptyset$, then $G_k(A)$ and $G_k(B)$ are independent.

We define a Gaussian process $\BY_k = \bigl( Y_k(t), \, t\geq0 \bigr)$ by
$$
Y_k(t) = \int_{(\bbr^d)^{k+1}} h_t(0,\by) G_k(d\by), \ \ t\geq 0,
$$
where $h_t$ is given in \eqref{e:def.ht}.
This process involves the same indicator function as  $\BV_k$, which implies that, similarly to the last regime, the $k$-dimensional holes affecting $\BY_k$ must be always formed by connected components on $k+2$ vertices (i.e., components of the smallest size).

We now state the main limit theorem for $L_{k,n}(t)$. The proof is presented in the Appendix.
\begin{theorem}  \label{t:sparse.fclt}
Suppose that $(R_n)$ satisfies \eqref{e:second.Rn}. Then,
$$
\bigl( n^{k+2}R_n^d f(R_ne_1)^{k+2} \bigr)^{-1/2} \Bigl( L_{k,n}(t) - \E \bigl\{ L_{k,n}(t) \bigr\}\Bigr) \Rightarrow \int_0^t Y_k(s)ds \ \ \text{in } C[0,\infty).
$$
\end{theorem}
\begin{remark}  \label{r:betti.conv.second}
This theorem does not mention anything about a direct result on the FCLT for $\beta_{k,n}(t)$. As can be seen in the proof of the theorem, however, a slight modification of the argument proves the CLT for $\beta_{k,n}(t)$ in a finite-dimensional sense. Namely, under the assumptions of Theorem \ref{t:sparse.fclt},
$$
\beta_{k,n}(t) \stackrel{fidi}{\Rightarrow}  Y_k(t),
$$
where $\stackrel{fidi}{\Rightarrow}$ denotes a finite-dimensional weak convergence. We believe that this holds even in the space $D[0,\infty)$ of right-continuous functions with left limits, but we are unable to prove the required tightness.
\end{remark}

In order to further clarify the structure of $\BY_k$, we express the process as
\begin{align*}
Y_k(t) &= \int_{(\bbr^d)^{k+1}} h_t^+(0,\by) G_k(d\by) - \int_{(\bbr^d)^{k+1}} h_t^-(0,\by) G_k(d\by) \\
&:= Y_k^+(t) - Y_k^-(t).
\end{align*}
We claim that $\BY_k^+$ and $\BY_k^-$ are represented as a time-changed Brownian motion. Note, however, that, although $\BY_k$ is a Gaussian process, it cannot be denoted as a (time-changed) Brownian motion.
\begin{proposition}
The process $\BY_k^{\pm}$ can be represented in law as
$$
\bigl( Y_k^{\pm}(t), \, t\geq0 \bigr)\stackrel{d}{=} \Bigl( B^{\pm} \bigl( D_k^{\pm}\, t^{d(k+1)} \bigr), \, t\geq0 \Bigr),
$$
where $B^{\pm}$ denotes the standard Brownian motion, and $D_k^{\pm} := C_k\int_{(\bbr^d)^{k+1}}h^{\pm}(0,\by)\, d\by$.
\begin{proof}
It suffices to prove that the covariance functions on both sides coincide. It follows from \eqref{e:increase.ht} that for $0\leq s\leq t$,
\begin{align*}
\E \bigl\{ Y_k^{\pm}(t)Y_k^{\pm}(s) \bigr\} &= C_k\int_{(\bbr^d)^{k+1}} \hspace{-10pt} h_t^{\pm}(0,\by)\, h_s^{\pm}(0,\by)\, d\by \\
&=s^{d(k+1)} D_k^{\pm} \\
&= \E \Bigl\{ B^{\pm}\bigl(D_k^{\pm}\, t^{d(k+1)}\bigr)\, B^{\pm}\bigl(D_k^{\pm}\, s^{d(k+1)}\bigr) \Bigr\}.
\end{align*}
\end{proof}
\end{proposition}

\subsection{Functional Central Limit Theorem in the Third Regime}
\label{s:third.regime}

Finally, we turn to the third regime in which $(R_n)$ is determined by
\begin{equation}  \label{e:Rn.regime3.no}
nf(R_ne_1) \to \lambda \ \ \text{as } n\to\infty
\end{equation}
for some $\lambda>0$.
In this case, the formation of $k$-dimensional holes drastically varies as compared to the previous regimes.
If $(R_n)$ satisfies \eqref{e:Rn.regime3.no}, then, by definition, $B(0; R_n)$ coincides with the weak core (up to multiplicative factors). Therefore, many random points become highly connected to one another in the area sufficiently close to the weak core. As a result, connected components on $i$ vertices for $i=k+2,k+3,\dots$ can all contribute to the limit in the FCLT. This phenomenon was never observed in the previous regimes.

In order to make the notations for defining the limiting process significantly lighter, we introduce several shorthand notations.
First, for $x_i \in \bbr^d$, $i=1,\dots,m$, and $r>0$,
$$
\B (x_1,\dots,x_m; r) := \bigcup_{i=1}^m B(x_i; r).
$$
For $i \geq k+2$, $j\geq1$, and $t\geq0$, we define an indicator $h_t^{(i,j)}: (\bbr^d)^i \to \{ 0,1 \}$ by
\begin{equation}  \label{e:def.hijt}
\hij_t(\Y) := \one \Bigl\{\, \beta_k \bigl( \check{C}(\Y; t) \bigr) = j, \ \check{C}(\Y; t) \text{ is connected} \Bigr\}\,, \ \ \ \Y = (y_1,\dots,y_i) \in (\bbr^d)^i.
\end{equation}
Clearly, $h_t^{(k+2,1)}$ coincides with the $h_t$ defined in \eqref{e:def.ht}. In particular, we write $\hij(\Y) := \hij_1(\Y)$. \\
Furthermore, for $i, \ip \geq k+2$, $j, \jp \geq 1$, and $t,s \geq0$, define an indicator $\hijipjp_{t,s}: (\bbr^d)^{i + \ip} \to \{ 0,1 \}$ by
$$
\hijipjp_{t,s} (\Y, \Yp) = \hij_t(\Y)\, \hipjp_s(\Yp), \ \ \ \Y \in (\bbr^d)^i, \ \Yp \in (\bbr^d)^{\ip},
$$
and, we set, for $i, \ip \geq k+2$, $t,s\geq0$,
$$
\Diip (t,s) := \bigl\{ (x_1,\dots,x_{i+\ip}) \in (\bbr^d)^{i+\ip}: \B(x_1,\dots,x_i; t) \cap \B (x_{i+1},\dots, x_{i+\ip}; s) \neq \emptyset \bigr\}.
$$
In the special case $t=s$, we denote $\Diip (t) := \Diip (t,t)$.

Now, we define stochastic processes $\BZ_k^{(i,j)} = \bigl(\Zij_k(t), \, t\geq 0  \bigr)$ for $i \geq k+2$ and $j\geq 1$, which function as the building blocks for the limiting process in the FCLT. First, define, for $i, \ip \geq k+2$, $j, \jp \geq 1$, $t,s \geq 0$, and $\lambda > 0$,
\begin{align}
\muijjp_k (t,s,\lambda) &:= s_{d-1} \int_1^\infty \rho^{d-1-\alpha i} \int_{(\bbr^d)^{i-1}} \hspace{-5pt} \hij_t(0,\by)\, \hijp_s(0,\by)  \label{e:def.muijjp}  \\
& \qquad \quad \times e^{-\lambda \rho^{-\alpha} (s \vee t)^d \text{vol} \bigl( \B (0,\by; 1) \bigr)} d\by d\rho, \notag
\end{align}
and
\begin{align}
\xiijipjp_k(t,s,\lambda) &:= s_{d-1} \int_1^\infty \rho^{d-1-\alpha (i+\ip)} \int_{(\bbr^d)^{i+\ip-1}} \hspace{-5pt} \hijipjp_{t,s}(0,\by) \label{e:def.xiijipjp} \\
&  \qquad \times \biggl[ \Bigl( \one_{\Diip (t,s)}(0,\by) - \one_{\Diip( (t\vee s) /2)} (0,\by) \Bigr)\, \notag \\
&\qquad \times e^{-\lambda \rho^{-\alpha} \text{vol} \bigl( \B(0,y_1,\dots y_{i-1}; t) \cup \B (y_i, \dots, y_{i+\ip-1}; s) \bigr)}- \one_{\Diip (t,s)}(0,\by)   \notag \\
& \qquad  \times e^{-\lambda \rho^{-\alpha} \bigl[ \text{vol}\bigl(\B(0,y_1,\dots y_{i-1}; t)\bigr) +  \text{vol}\bigl(\B(y_i,\dots y_{i+\ip-1}; s)\bigr) \bigr]}   \biggr] d\by d\rho, \notag
\end{align}
where $a \vee b = \max \{ a, b \}$ for $a, b \in \bbr$, and $h_t^{(i,j)}(0,\by) = h_t^{(i,j)}(0,y_1,\dots,y_{i-1})$ with $0, y_1,\dots,y_{i-1}\in \bbr^d$ etc.
These functions are used to formulate the covariance functions of $\BZ^{(i,j)}_k$'s.
More specifically, for $i \geq k+2$ and $j \geq 1$, we define $\BZ_k^{(i,j)}$ as a zero-mean Gaussian process with the covariance function given by
\begin{equation}  \label{e:cov.Zijk}
\text{Cov} \bigl( \Zij_k(t), \Zij_k(s) \bigr)
= \frac{\lambda^i}{i!}\, \muijj_k(t,s,\lambda) + \frac{\lambda^{2i}}{(i!)^2}\, \xi^{(i,j,i,j)}_k(t,s,\lambda), \ \ t,s \geq 0.
\end{equation}
For every $i \geq k+2$, there exists $j_0  \geq 1$, which depends on $i$, such that for all $j \geq j_0$ and $t\geq0$, $\hij_t$ is identically zero, in which case, \eqref{e:cov.Zijk} allows us to take $\BZ_k^{(i,j)}$ as a zero process, i.e., $\Zij_k(t) \equiv 0$ for all $t\geq0$. For example, $\BZ_k^{(k+2,j)}$ is a zero process for all $j \geq 2$.

In addition, we assume that the processes $\bigl( \BZ_k^{(i,j)}, \, i \geq k+2, \, j\geq1 \bigr)$ are dependent on each other in such a way that for $i, \ip \geq k+2$, $j, \jp \geq 1$,
\begin{align*}
&\text{Cov} \bigl( \Zij_k(t), \Zipjp_k(s) \bigr)
= \frac{\lambda^i}{i!}\, \muijjp_k(t,s,\lambda)\, \delta_{i,\ip} + \frac{\lambda^{i+\ip}}{i!\, \ip !}\, \xiijipjp_k(t,s,\lambda), \ \ \ t,s \geq 0,
\end{align*}
where $\delta_{i,\ip}$ is the Kronecker delta.

We now  define a zero-mean Gaussian process by
\begin{equation}  \label{e:def.Zkt}
Z_k(t) := \sum_{i=k+2}^\infty \sum_{j\geq1} j \Zij_k(t), \ \ \ t\geq0,
\end{equation}
which appears in the limiting process in the FCLT.
It is shown in the proof of Theorem \ref{t:giant.fclt} below that the right hand side of \eqref{e:def.Zkt}  almost surely converges for each $t\geq0$.

We can rewrite $Z_k(t)$ as
$$  
Z_k(t) = Z_k^{(k+2,1)}(t) + \sum_{i=k+3}^\infty \sum_{j\geq1} j \Zij_k(t).
$$
Since the covariance function of $\BZ_k^{(i,j)}$ involves the indicator function $\hij_\cdot$, we can consider the process $\BZ_k^{(i,j)}$ as representing the connected components that are on $i$ vertices and possess $j$ holes. In particular, the process $\BZ_k^{(k+2,1)}$ represents the connected components on $k+2$ vertices with a single $k$-dimensional hole. This implies that $\BZ_k^{(k+2,1)}$ may share the same property as $\BY_k$ in the last regime in the sense that both processes represent connected components only of the smallest size.
In the present regime, however, we cannot ignore the effect of larger components emerging near the weak core, and therefore, many other Gaussian processes, except for $\BZ_k^{(k+2,1)}$, will contribute to the limit in the FCLT.

Before presenting the main limit theorem, we add a technical assumption that a constant $\lambda$ in \eqref{e:Rn.regime3.no} is less than $(e\omega_d)^{-1}$, where $\omega_d$ is the volume of a unit ball in $\bbr^d$. It seems that the FCLT below still holds without any upper bound condition for $\lambda$, but this is needed for technical reasons during the proof. Similarly, the domain of functions in the space $C$ must be restricted to the unit interval $[0,1]$.
The proof of the theorem is deferred to the Appendix.
\begin{theorem}  \label{t:giant.fclt}
Suppose that $(R_n)$ satisfies
\begin{equation}  \label{e:Rn.regime3}
nf(R_ne_1) \to \lambda \in \bigl(0, (e\omega_d)^{-1} \bigr) \ \ \text{as } n\to\infty.
\end{equation}
Then,
$$
R_n^{-d/2} \Bigl( L_{k,n}(t) - \E \bigl\{ L_{k,n}(t) \bigr\}\Bigr) \Rightarrow \int_0^t Z_k(s)ds \ \ \text{in } C[0,1].
$$
\end{theorem}
\begin{remark}
As in Remark \ref{r:betti.conv.second}, we can also obtain finite-dimensional convergence of $\beta_{k,n}(t)$. That is, under the conditions of Theorem \ref{t:giant.fclt},
$$
\beta_{k,n}(t) \stackrel{fidi}{\Rightarrow} Z_k(t).
$$
\end{remark}

\section{Appendix}  \label{s:appendix}

In this Appendix, we provide the proofs of Theorems \ref{t:sparse.poisson}, \ref{t:sparse.fclt}, and \ref{t:giant.fclt}. We first introduce the results known as the ``Palm theory" in order to compute the expectations related to Poisson point processes. Indeed, the Palm theory applies many times hereafter in the Appendix. In Section \ref{s:proof.first.regime}, we prove Theorem \ref{t:sparse.poisson}, and, subsequently, in Section \ref{s:proof.third.regime} we verify Theorem \ref{t:giant.fclt}. We give the proof of  Theorem \ref{t:sparse.fclt} in Section \ref{s:proof.second.regime}, while exploiting many of the results established in the former Section \ref{s:proof.third.regime}.

Before proceeding to specific subsections, we introduce some useful shorthand notations to save space. For $\bx = (x_1,\dots,x_m) \in (\bbr^d)^m$, $x \in \bbr^d$, and $\by = (y_1,\dots,y_{i-1})
\in (\bbr^d)^{i-1}$,
\begin{align*}
f(\bx) &:= f(x_1)\cdots f(x_m)\,, \\
f(x+\by) &:= f(x + y_1)\cdots f(x + y_{m-1})\,, \\
\hij (0,\by) &:= \hij (0,y_1,\dots,y_{i-1}) \ \ \text{etc.}
\end{align*}
Denote also by $C^*$ a generic positive constant, which can vary between lines and is independent of $n$.

\subsection{Palm Theory}

\begin{lemma} (Palm theory for Poisson point processes, \cite{arratia:goldstein:gordon:1989}, see also Section 1.7 in \cite{penrose:2003})  \label{l:palm}
Let $(X_i)$ be iid $\bbr^d$-valued random variables with common density $f$. Let $\Pn$ be a Poisson point process on $\bbr^d$ with intensity $nf$. Let $u(\Y, \mathcal X)$ and $v(\Yp, \X)$ be measurable bounded functions defined for $\Y \in (\bbr^d)^\ell$, $\Yp \in (\bbr^d)^m$, and a finite subset $\X  \supset \Y, \Yp$ of $d$-dimensional real vectors. Then,
\begin{align*}
\E \biggl\{ \sum_{\Y \subset \Pn}  u(\Y, \Pn) \biggr\} &= \frac{n^\ell}{\ell !}\, \E \bigl\{ u(\Yp, \Yp \cup \Pn) \bigr\}\,,
\end{align*}
where $\Yp$ is a set of $\ell$ iid points in $\bbr^d$ with density $f$, independent of $\Pn$.
Furthermore,
\begin{align*}
\E \biggl\{ \sum_{\Y \subset \Pn} \sum_{\Yp \subset \Pn, \, |\Y \cap \Yp| = 0} \hspace{-5pt} u(\Y, \Pn)\, v(\Yp, \Pn)\, \biggr\}
= \frac{n^{\ell + m}}{\ell ! \, m !}\, \E \Bigl\{u(\Y_1, \Y_{12} \cup \Pn)\, v(\Y_2, \Y_{12} \cup \Pn)\, \Bigr\}\,,
\end{align*}
where $\Y_1$ is a set of $\ell$ iid points in $\bbr^d$ and $\Y_2$ is a set of $m$ iid points in $\bbr^d$, such that $\Y_{12} := \Y_1 \cup \Y_2$ is independent of $\Pn$, and $|\Y_1 \cap \Y_2| = 0$, that is, there are no common points between $\Y_1$ and $\Y_2$.
\vspace{5pt}

Moreover, let $w_i(\Y)$, $i=1,2$ be measurable bounded functions defined for $\Y \in (\bbr^d)^p$. Then, for every $q \in \{ 0,\dots, p \}$,
\begin{align*}
\E \biggl\{  \, \sum_{\Y \subset \Pn} \sum_{\Yp \subset \Pn, \, |\Y \cap \Yp| = q} \hspace{-5pt} w_1(\Y)\, w_2(\Yp)\, \biggr\} = \frac{n^{2p-q}}{q! \bigl( (p-q)! \bigr)^2}\, \E \bigl\{  w_1(\Y_1)\, w_2(\Y_2) \bigr\},
\end{align*}
where $\Y_1$ and $\Y_2$ are sets of $p$ iid points in $\bbr^d$ with $|\Y_1 \cap \Y_2| = q$.
\end{lemma}


\subsection{Proof of Theorem \ref{t:sparse.poisson}}  \label{s:proof.first.regime}

Since \eqref{e:limit.barcode.first} immediately follows from \eqref{e:limit.betti.first} by the continuous mapping theorem, we may prove only \eqref{e:limit.betti.first}.
The proof of \eqref{e:limit.betti.first} is divided into two parts. In the first, we show that
\begin{align}
G_{k,n}(t) := \sum_{\Y \subset \Pn} h_t(\Y)\, \one \bigl\{ m(\Y) \geq R_n \bigr\} \Rightarrow V_k(t) \ \ \text{in } D[0,\infty),  \label{e:first.sparse.poisson}
\end{align}
where $m(x_1,\dots,x_{k+2}) = \min_{1 \leq i \leq k+2} \| x_i \|$, $x_i \in \bbr^d$, and, in the second, we prove that the difference between $G_{k,n}(t)$ and $\beta_{k,n}(t)$ vanishes in probability in the space $D[0,\infty)$.
\begin{proof}
\noindent \underline{\textit{Part I}} We begin with the finite-dimensional weak convergence of \eqref{e:first.sparse.poisson}, for which we need to verify
\begin{equation}  \label{e:CWdevice}
\sum_{j=1}^m a_j G_{k,n}(t_j) \Rightarrow \sum_{j=1}^m a_j V_k(t_j)
\end{equation}
for every $a_1, \dots, a_m \in \bbr$, $t_1, \dots, t_m \geq 0$, $m \geq 1$.

Let $\sum_\ell \epsilon_{v_\ell}$ denote a Poisson random measure on $\bbr$ with finite mean measure
$$
C_k \int_{(\bbr^d)^{k+1}} \one \Bigl\{ \, \sum_{j=1}^m a_j h_{t_j}(0,\by) \in \cdot \setminus \{ 0 \}\, \Bigr\}d\by
$$
(``$\epsilon$" represents the usual Dirac measure). It is then elementary to verify that
$$
\sum_{\ell} v_\ell \stackrel{d}{=} \sum_{j=1}^m a_j V_k(t_j).
$$
Writing $M_p(\bbr)$ for the space of point measures on $\bbr$, \eqref{e:CWdevice} will be complete, provided that we can show the point process convergence
\begin{align}
\xi_n &:= \sum_{\Y \subset \Pn} \one \Bigl\{ \, \sum_{j=1}^m a_j h_{t_j}(\Y) \neq 0, \ m(\Y) \geq R_n\, \Bigr\}\, \epsilon_{\Bigl(\, \sum_{j=1}^m a_j h_{t_j}(\Y)\, \Bigr)} \label{e:point.proc.conv} \\
&\Rightarrow \sum_\ell \epsilon_{v_\ell} \ \ \text{in } M_p(\bbr). \notag
\end{align}
Indeed, since the functional $\widehat T: M_p(\bbr) \to \bbr$ defined by $\widehat T \bigl( \sum_\ell \epsilon_{z_\ell} \bigr) = \sum_\ell z_\ell$ is continuous on a set of \textit{finite} point measures, \eqref{e:point.proc.conv} implies \eqref{e:CWdevice} by the continuous mapping theorem.

According to \cite{decreusefond:schulte:thaele:2015} (or use Theorem 2.1 in \cite{owada:adler:2016}), in order to establish \eqref{e:point.proc.conv}, it suffices to prove the following results: as $n\to\infty$,
\begin{equation}  \label{e:1st.cond.pp}
\E \bigl\{ \xi_n(A) \bigr\} \to \E \Bigl\{ \sum_\ell \epsilon_{v_\ell}(A) \Bigr\} \ \ \text{for every measurable } A \subset (\bbr^d)^{k+1},
\end{equation}
and
\begin{align}
r_n := \max_{1 \leq \ell \leq k+1} n^{2k+4-\ell} \P &\Bigl\{ \ \sum_{j=1}^m a_j h_{t_j}(X_1,\dots,X_{k+2}) \neq 0, \label{e:2nd.cond.pp} \\
&\quad \sum_{j=1}^m a_j h_{t_j}(X_1,\dots,X_\ell, X_{k+3}, \dots, X_{2k+4-\ell}) \neq 0, \notag \\
&\quad \|X_i \| \geq R_n, \, i=1,\dots,2k+4-\ell\, \Bigr\} \to 0. \notag
\end{align}
For the proof of \eqref{e:1st.cond.pp}, it follows from the Palm theory in Lemma \ref{l:palm} that
\begin{align*}
\E \bigl\{ \xi_n(A) \bigr\} &= \frac{n^{k+2}}{(k+2)!}\, \int_{(\bbr^d)^{k+2}} f(\bx)\, \one \bigl\{ m(\bx) \geq R_n \bigr\}\, \one \Bigl\{ \ \sum_{j=1}^m a_j h_{t_j}(\bx) \in A \setminus \{ 0 \}\, \Bigr\} d\bx.
\end{align*}
Changing the variables $x_1 \leftrightarrow x$, $x_\ell \leftrightarrow x+y_{\ell-1}$, $\ell = 2,\dots,k+2$, together with the location invariance of $h_{t_j}$'s,
\begin{align}
\E \bigl\{ \xi_n(A) \bigr\} = \frac{n^{k+2}}{(k+2)!}\, \int_{\bbr^d}\int_{(\bbr^d)^{k+1}} &f(x)\, f(x+\by)\, \one \bigl\{ \,m(x,x+\by) \geq R_n\, \bigr\} \label{e:1st.change} \\
&\times \one \Bigl\{ \ \sum_{j=1}^m a_j h_{t_j}(0, \by) \in A \setminus \{ 0 \}\, \Bigr\} d\by dx. \notag
\end{align}
The polar coordinate transform $x \leftrightarrow (r,\theta)$, followed by an additional change of variable $r \leftrightarrow R_n\rho$, yields
\begin{align}
\E \bigl\{ \xi_n(A) \bigr\} &= \frac{n^{k+2}R_{k,n}^df(R_{k,n}e_1)^{k+2}}{(k+2)!} \int_{S_{d-1}} J(\theta) d\theta \int_1^{\infty} d\rho \int_{(\bbr^d)^{k+1}} \hspace{-5pt} d\by \,  \label{e:polar1} \\
&\qquad \times \rho^{d-1} \frac{f(R_n\rho e_1)}{f(R_ne_1)}\, \prod_{\ell = 1}^{k+1}\, \frac{f \bigl( R_n \| \rho \theta + y_\ell /R_n \|e_1 \bigr)}{f(R_ne_1)}\, \one \bigl\{ \| \rho \theta + y_\ell /R_n \| \geq 1 \bigr\} \notag \\
&\qquad \times \one \Bigl\{ \ \sum_{j=1}^m a_j h_{t_j}(0, \by) \in A \setminus \{ 0 \}\, \Bigr\}, \notag
\end{align}
where $S_{d-1}$ is the $(d-1)$-dimensional unit sphere in $\bbr^d$ and $J(\theta)$ is the usual Jacobian, that is,
$$
J(\theta)  =   \sin^{k-2}(\theta_1) \, \sin^{k-3}(\theta_2) \, \cdots\, \sin (\theta_{k-2}).
$$
By the regular variation assumption \eqref{e:RV.tail} of $f$, we have that for every $\rho \geq 1$, $\theta \in S_{d-1}$, and $y_1,\dots, y_{k+1} \in \bbr^d$,
$$
\frac{f(R_n\rho e_1)}{f(R_ne_1)}\, \prod_{\ell = 1}^{k+1}\, \frac{f \bigl( R_n \| \rho \theta + y_\ell /R_n \|e_1 \bigr)}{f(R_ne_1)} \to \rho^{-\alpha (k+2)}, \ \ \ n\to\infty.
$$
Therefore, supposing the dominated convergence theorem is applicable, we can obtain
\begin{align*}
\E \bigl\{ \xi_n(A) \bigr\} &\to C_k \int_{(\bbr^d)^{k+1}} \one \Bigl\{ \, \sum_{j=1}^m a_j h_{t_j}(0,\by) \in A \setminus \{ 0 \}\, \Bigr\}d\by
\\
&= \E \Bigl\{ \sum_\ell \epsilon_{v_\ell} (A) \Bigr\}, \ \ \ n\to\infty.
\end{align*}

To establish an integrable upper bound, we use the so-called Potter's bound (e.g., Proposition 2.6 $(ii)$ in \cite{resnick:2007}); for every $0 < \xi < \alpha-d$, we have 
\begin{equation}  \label{e:P-Potter1}
\frac{f(R_n\rho e_1)}{f(R_ne_1)}\, \one \{  \rho \geq 1\} \leq (1+\xi) \rho^{-\alpha +\xi}  \one \{  \rho \geq 1\},
\end{equation}
\begin{equation}  \label{e:P-Potter2}
\prod_{\ell = 1}^{k+1}\, \frac{f \bigl( R_n \| \rho \theta + y_\ell /R_n \|e_1 \bigr)}{f(R_ne_1)}\, \one \bigl\{ \| \rho \theta + y_\ell /R_n \| \geq 1 \bigr\} \leq (1+\xi)^{k+1}
\end{equation}
for sufficiently large $n$. Since $\int_1^\infty \rho^{d-1-\alpha +\xi}d\rho < \infty$, the dominated convergence theorem applies as required.

As for the proof of \eqref{e:2nd.cond.pp}, proceeding by changing the variables in the same way as the previous argument, we see that as $n\to\infty$,
$$
r_n = \max_{1 \leq \ell \leq k+1} \mathcal O \bigl(n^{2k+4-\ell} R_{k,n}^d f(R_{k,n}e_1)^{2k+4-\ell} \bigr) = \max_{1 \leq \ell \leq k+1} \mathcal O \Bigl(\bigl(n  f(R_{k,n}e_1)\bigr)^{k+2-\ell} \Bigr) \to 0.
$$
Now, the claim is proved.

Next, we show the tightness of $\bigl( G_{k,n}(t), \, t\geq 0 \bigr)$ in the space $D[0,\infty)$ equipped with the Skorohod $J_1$-topology. By Theorem 13.4 in \cite{billingsley:1999}, it suffices to show that for every $L>0$, there exists $B>0$ such that
$$
\P \biggl\{ \min \Bigl\{ \bigl|  G_{k,n}(t) - G_{k,n}(s) \bigr|, \, \bigl|  G_{k,n}(s) - G_{k,n}(r) \bigr| \Bigr\} \geq \lambda \biggr\} \leq B \lambda^{-2} (t-r)^2
$$
for all $0 \leq r \leq s \leq t \leq L$, $n \geq 1$, and $\lambda > 0$. For typographical ease, define for $n \geq 1$ and $0 \leq s \leq t$,
\begin{align}
h_{n,t}(\Y) := h_t (\Y)\, \one \bigl\{  m(\Y) \geq R_{k,n} \bigr\}, \ \ \Y \in (\bbr^d)^{k+2}, \notag \\
h_{t,s}(\Y) := h_t (\Y) - h_s(\Y), \ \ \Y \in (\bbr^d)^{k+2}, \notag \\
h_{t,s}^{\pm}(\Y) := h_t^{\pm} (\Y) - h_s^{\pm}(\Y), \ \ \Y \in (\bbr^d)^{k+2}. \label{e:def.htspm}
\end{align}

By Markov's inequality, we only have to show that
\begin{equation}  \label{e:tightness.poisson}
\E \biggl\{ \, \sum_{\Y \subset \Pn} \sum_{\Yp \subset \Pn} \bigl| h_{n,t}(\Y) - h_{n,s}(\Y) \bigr|\, \bigl| h_{n,s}(\Yp) - h_{n,r}(\Yp) \bigr| \biggr\} \leq B(t-r)^2
\end{equation}
for all $0 \leq r \leq s \leq t \leq L$ and $n \geq 1$. The left hand side above is clearly equal to
\begin{align*}
&\sum_{\ell = 0}^{k+2} \E \biggl\{ \, \sum_{\Y \subset \Pn} \sum_{\Yp \subset \Pn, \, |\Y \cap \Yp| = \ell} \bigl| h_{n,t}(\Y) - h_{n,s}(\Y) \bigr|\, \bigl| h_{n,s}(\Yp) - h_{n,r}(\Yp) \bigr|\, \biggr\} := \sum_{\ell =0}^{k+2} \E \{ I_{n,\ell} \}.
\end{align*}
For $\ell = 1,\dots,k+2$, the Palm theory yields
$$
\E \{ I_{n,\ell} \} = \frac{n^{2k+4-\ell}}{\ell !\, \bigl( (k+2-\ell)! \bigr)^2}\, \E \Bigl\{ \bigl|  h_{n,t}(\Y_1) - h_{n,s}(\Y_1) \bigr|\,  \bigl|  h_{n,s}(\Y_2) - h_{n,r}(\Y_2) \bigr|\,  \Bigr\},
$$
where $\Y_1$ and $\Y_2$  are sets of $(k+2)$ iid points in $\bbr^d$ sharing $\ell$ common points, that is, $|  \Y_1 \cap \Y_2 | = \ell$. By the same change of variables as in \eqref{e:1st.change} and \eqref{e:polar1}, together with \eqref{e:Rn.regime1} and Potter's bound, we eventually have
\begin{align*}
\E \{  I_{n,\ell}\} &\leq C^* \int_{(\bbr^d)^{\ell -1}}d \by \int_{(\bbr^d)^{k+2-\ell}}\hspace{-10pt} d \bz_2 \int_{(\bbr^d)^{k+2-\ell}}\hspace{-10pt}d \bz_1 \bigl| h_{t,s}(0,\by,\bz_1)  \bigr|\, \bigl| h_{s,r}(0,\by,\bz_2)  \bigr|  \\
&\leq C^* \int_{(\bbr^d)^{\ell -1}}d \by \int_{(\bbr^d)^{k+2-\ell}}\hspace{-10pt} d \bz_2 \int_{(\bbr^d)^{k+2-\ell}}\hspace{-10pt}d \bz_1 \\
&\qquad \qquad \qquad \quad \bigl( \, h_{t,s}^+(0,\by,\bz_1)\, h_{s,r}^+(0,\by,\bz_2) + h_{t,s}^-(0,\by,\bz_1)\, h_{s,r}^-(0,\by,\bz_2) \\
&\qquad \qquad \qquad \qquad + h_{t,s}^+(0,\by,\bz_1)\, h_{s,r}^-(0,\by,\bz_2) + h_{t,s}^-(0,\by,\bz_1)\, h_{s,r}^+(0,\by,\bz_2) \, \bigr)
\end{align*}
Applying Lemma \ref{l:tightness.lemma} below, the rightmost term is bounded by $C^* (t-r)^2$, as required.

We need to establish a suitable upper bound for $\E \{ I_{n,0} \}$ as well. By the Palm theory,
$$
\E \{  I_{n,0}\} = \frac{n^{2k+4}}{\bigl( (k+2)! \bigr)^2}\, \E \Bigl\{  \bigl| h_{n,t}(\Y) - h_{n,s}(\Y)  \bigr| \Bigr\}\, \E \Bigl\{  \bigl| h_{n,s}(\Y) - h_{n,r}(\Y)  \bigr| \Bigr\},
$$
and the same argument as above can provide an upper bound of the form $C^*(t-r)^2$. Now, we can conclude \eqref{e:tightness.poisson}.
\vspace{5pt}

\noindent \underline{\textit{Part II}} To complete the proof, one needs to show that
\begin{equation}  \label{e:goal.part2}
 \beta_k \Bigl( \check{C}\bigl( \Pn \cap B(0;R_{k,n})^c; t \bigr) \Bigr) - G_{k,n}(t) \stackrel{p}{\to} 0 \ \ \text{in } D[0,\infty).
\end{equation}
To this end, we use obvious inequalities
$$
G_{k,n}(t) \leq \beta_k \Bigl( \check{C}\bigl( \Pn \cap B(0;R_{k,n})^c; t \bigr) \Bigr) \leq G_{k,n}(t) + L_{k,n}(t),
$$
where
$$
L_{k,n}(t) = \sum_{\Y \subset \Pn} \one \bigl\{  |\Y| = k+3, \ \check{C}(\Y; t) \text{ is connected}  \bigr\} \times \one \bigl\{  m(\Y) \geq R_{k,n} \bigr\}
$$
with $m(x_1,\dots, x_{k+3}) = \min_{1 \leq i \leq k+3} \| x_i \|$, $x_i \in \bbr^d$. \\
We have, for every $T>0$,
\begin{align*}
\E &\biggl\{  \, \sup_{0 \leq t \leq T} \Bigl[\,  \beta_k \Bigl( \check{C}\bigl( \Pn \cap B(0;R_{k,n})^c; t \bigr) \Bigr) - G_{k,n}(t) \Bigr]  \biggr\} \leq  \E \bigl\{ \, \sup_{0 \leq t \leq T} L_{k,n}(t) \bigr\}  \\
&\leq \frac{n^{k+3}}{(k+3)!}\, \P \Bigl\{  \check{C}(X_1,\dots,X_{k+3}; T) \text{ is connected}, \ \| X_i \| \geq R_{k,n}, \, i =1,\dots,k+3\,  \Bigr\}
\end{align*}
The same change of variables as in \eqref{e:1st.change} and \eqref{e:polar1}, together with Potter's bound, concludes that the rightmost term above turns out to be
$$
\mathcal O \bigl( n^{k+3} R_{k,n}^d f(R_{k,n}e_1)^{k+3} \bigr) = \mathcal O \bigl( n f(R_{k,n}e_1) \bigr) \to 0 \ \ \text{as } n \to \infty.
$$
Thus, \eqref{e:goal.part2} follows.
\end{proof}
\begin{lemma}  \label{l:tightness.lemma}
Let $h_t, h_t^{\pm} : (\bbr^d)^{k+2} \to \{ 0,1 \}$ be indicator functions given in \eqref{e:def.ht} and \eqref{e:decomp.ind}, and recall notation \eqref{e:def.htspm}. Fix $L>0$. Then, we have, for $\ell \in \{ 1,\dots,k+2 \}$,
\begin{align*}
\int_{(\bbr^d)^{\ell -1}}d \by &\int_{(\bbr^d)^{k+2-\ell}}\hspace{-10pt} d \bz_2 \int_{(\bbr^d)^{k+2-\ell}}\hspace{-10pt}d \bz_1 \Bigl( h_{t,s}^+(0,\by,\bz_1)\, h_{s,r}^+(0,\by,\bz_2)  + h_{t,s}^-(0,\by,\bz_1)\,h_{s,r}^-(0,\by,\bz_2)    \\
&\qquad \quad  + h_{t,s}^+(0,\by,\bz_1)\, h_{s,r}^-(0,\by,\bz_2) + h_{t,s}^-(0,\by,\bz_1)\, h_{s,r}^+(0,\by,\bz_2) \Bigr) \leq C^*(t-r)^2
\end{align*}
for all $0 \leq r \leq s \leq t \leq L$.
\end{lemma}
\begin{proof}
Let $I_1 + I_2 + I_3 + I_4$ denote the triple integral on the left hand side. It follows from Lemma 7.1 in \cite{owada:2016} that $I_i \leq C^* (t-r)^2$ for $i=1,2$. The same argument can yield $I_i \leq C^* (t-r)^2$ for $i=3,4$ as well.
\end{proof}

\subsection{Proof of Theorem \ref{t:giant.fclt}}  \label{s:proof.third.regime}

The goal of this subsection is to complete the proof of Theorem \ref{t:giant.fclt}. The proof is, however, rather long, and therefore it is divided into several parts.

First, we define for $i \geq k+2$, $j \geq 1$, $t\geq0$, and $n\geq1$,
$$
\hij_{n,t}(\Y) := h_t^{(i,j)}(\Y)\, \one \bigl\{ m(\Y) \geq R_n \bigr\}, \ \ \ \Y \in (\bbr^d)^i,
$$
where $h_t^{(i,j)}$ is given in \eqref{e:def.hijt}, and $m(x_1,\dots,x_m) = \min_{1 \leq \ell \leq m} \| x_\ell \|$, $x_1,\dots, x_m \in \bbr^d$, $m \geq 1$.
Next, define for $i \geq k+2$, $j\geq1$, $t\geq 0$, $\Y \in (\bbr^d)^i$, and a finite subset of $d$-dimensional real vectors $\mathcal Z \supset \Y$
$$
\gij_t(\Y, \mathcal Z) := \hij_t(\Y)\, \one \bigl\{ \check{C}(\Y; t) \text{ is an isolated component of } \check{C}(\mathcal Z; t)  \bigr\}
$$
and
$$
\gij_{n,t}(\Y, \mathcal Z) := \gij_t(\Y, \mathcal Z)\, \one \bigl\{ m(\Y) \geq R_n \bigr\}.
$$
Throughout the proof, we rely on a useful representation for the $k$-th Betti number adopted in \cite{kahle:meckes:2015}
$$
\beta_{k,n}(t) = \sum_{i=k+2}^\infty \sum_{j\geq1} j \sum_{\Y \subset \Pn} \gij_{n,t}(\Y, \Pn).
$$
Let Ann$(K_1,K_2)$ be an annulus of inner radius $K_1$ and outer radius $K_2$. For $x_1,\dots,x_m \in \bbr^d$, $m \geq 1$, define Max$(x_1,\dots,x_m)$ as the function selecting an element with largest distance from the origin. That is, Max$(x_1,\dots,x_m) = x_\ell$ if $\| x_\ell \| = \max_{1 \leq j \leq m}\| x_j \|$. If multiple $x_j$'s achieve the maximum, we choose an element with the smallest subscript.
The following quantity is associated with the $k$-th Betti number and plays an important role in our proof. For $1 \leq K \leq \infty$,
$$
\beta_{k,n}(t; K) := \sum_{i=k+2}^\infty \sum_{j\geq1} j \sum_{\Y \subset \Pn} \gij_{n,t}(\Y, \Pn)\, \one \bigl\{ \text{Max}(\Y) \in \text{Ann}(R_n, KR_n) \bigr\}.
$$
Clearly, $\beta_{k,n}(t; \infty) = \beta_{k,n}(t)$. Furthermore, we sometimes need a truncated Betti number
\begin{equation}  \label{e:truncated.betti}
\beta_{k,n}^{(M)}(t) := \sum_{i=k+2}^M \sum_{j\geq1} j \sum_{\Y \subset \Pn} \gij_{n,t}(\Y, \Pn).
\end{equation}
Analogously, we can also define $\beta_{k,n}^{(M)}(t; K)$ by the truncation.

We start with revealing the asymptotics of the mean and the covariance of the Betti numbers.
\begin{lemma}  \label{l:giant.cov}
For every $0 \leq  t, s \leq 1$ and $1\leq K\leq \infty$, we have, as $n\to \infty$,
$$
R_n^{-d}\, \E \bigl\{ \beta_{k,n}(t; K) \bigr\} \to \sum_{i=k+2}^\infty \sum_{j\geq1}\, j\, \frac{\lambda^i}{i!}\, \muijj_k (t,t,\lambda; K) \in (0,\infty),
$$
\begin{align*}
R_n^{-d}\, &\text{Cov} \bigl\{ \beta_{k,n}(t; K), \beta_{k,n}(s; K) \bigr\} \\
&\to C_k (t,s; K) := \sum_{i=k+2}^\infty \sum_{j, \jp \geq 1}\, j \jp\, \frac{\lambda^i}{i!}\, \muijjp_k (t,s,\lambda; K)  \\
&\qquad \qquad \qquad \quad +\sum_{i, \ip = k+2}^\infty \sum_{j, \jp \geq1} \, j \jp\, \frac{\lambda^{i + \ip}}{i! \, \ip!} \, \xiijipjp_k(t,s,\lambda; K) \in (-\infty,\infty)
\end{align*}
with
\begin{align}
\muijjp_k (t,s,\lambda; K) &:= s_{d-1} \int_1^K \rho^{d-1-\alpha i} \int_{(\bbr^d)^{i-1}} \hspace{-5pt} \hij_t(0,\by)\, \hijp_s(0,\by)  \label{e:def.muijjp1} \\
& \qquad \quad \times e^{-\lambda \rho^{-\alpha} (s \vee t)^d \text{vol} \bigl( \B (0,\by; 1) \bigr)} d\by d\rho, \notag
\end{align}
\begin{align*}
\xiijipjp_k(t,s,\lambda; K) &:= s_{d-1} \int_1^K \rho^{d-1-\alpha (i+\ip)} \int_{(\bbr^d)^{i+\ip-1}} \hspace{-5pt} \hijipjp_{t,s}(0,\by) \label{e:def.xiijipjp1} \\
&  \qquad \times \biggl[ \Bigl( \one_{\Diip (t,s)}(0,\by) - \one_{\Diip( (t\vee s) /2)} (0,\by) \Bigr)\, \notag \\
&\qquad \times e^{-\lambda \rho^{-\alpha} \text{vol} \bigl( \B(0,y_1,\dots y_{i-1}; t) \cup \B (y_i, \dots, y_{i+\ip-1}; s) \bigr)}- \one_{\Diip (t,s)}(0,\by)  \notag \\
& \qquad  \times e^{-\lambda \rho^{-\alpha} \bigl[ \text{vol}\bigl(\B(0,y_1,\dots y_{i-1}; t)\bigr) +  \text{vol}\bigl(\B(y_i,\dots y_{i+\ip-1}; s)\bigr) \bigr]}   \biggr] d\by d\rho. \notag
\end{align*}
In terms of notations \eqref{e:def.muijjp} and \eqref{e:def.xiijipjp}, we have $\muijjp_k(t,s,\lambda) = \muijjp_k(t,s,\lambda; \infty)$ and $\xiijipjp_k(t,s,\lambda) = \xiijipjp_k(t,s,\lambda; \infty)$.
\end{lemma}

To prove Lemma \ref{l:giant.cov}, we require the results for Lemmas \ref{l:conv.upper.lemma} and \ref{l:geo.lemma} below, for which we refine the ideas and techniques used in \cite{kahle:meckes:2015} and \cite{kahle:meckes:2013}. Without any loss of generality, we may prove only the case $K=\infty$.
\begin{proof}
By the monotone convergence theorem, together with the Palm theory in Lemma \ref{l:palm}, we have
$$
R_n^{-d}\, \E \bigl\{ \beta_{k,n}(t) \bigr\} = \sum_{i=k+2}^\infty \sum_{j\geq1} \, j\, R_n^{-d}\, \frac{ n^i}{i!}\, \E \bigl\{ \gij_{n,t}(\Yp, \Yp \cup \Pn) \bigr\},
$$
where $\Yp$ is a set of iid points in $\bbr^d$ with density $f$, independent of $\Pn$. \\
It follows from Lemma \ref{l:conv.upper.lemma} $(i)$ that
$$
R_n^{-d} n^i \E \bigl\{ \gij_{n,t}(\Yp, \Yp \cup \Pn) \bigr\} \to \lambda^i \muijj_k (t,t,\lambda).
$$
We need to justify the application of the dominated convergence theorem, for which we apply Lemma \ref{l:conv.upper.lemma} $(ii)$,  stating that there exists a positive integer $N \in \bbn_+$ so that for all $i \geq k+2$, $j \geq 1$, and $t\geq 0$,
$$
R_n^{-d} n^i \E \bigl\{ \gij_{n,t}(\Yp, \Yp \cup \Pn) \bigr\} \leq C^* \bigl( \lambda(1+\delta) \bigr)^i \int_{(\bbr^d)^{i-1}} \hij_t(0,\by) d\by \ \ \text{for all } n \geq N,
$$
where $\delta$ is a positive constant satisfying $\lambda (1+\delta) e \omega_d < 1$. \\
Appealing to Lemma \ref{l:geo.lemma} $(i)$, together with Stirling's formula $i! \geq (i/e)^i$ for sufficiently large $i$, we have
\begin{align*}
\sum_{i=k+2}^\infty \sum_{j\geq1}\, j\, \frac{\bigl( \lambda(1+\delta) \bigr)^i}{i!}\,  \int_{(\bbr^d)^{i-1}} \hij_t(0,\by) d\by  &\leq  \sum_{i=k+2}^\infty \frac{\bigl( \lambda(1+\delta) \bigr)^i}{i!}\, \begin{pmatrix} i \\ k+2 \end{pmatrix} i^{i-2} (\omega_d)^{i-1} \\
&\quad \leq C^* \sum_{i=k+2}^\infty i^k \bigl( \lambda (1+\delta) e \omega_d \bigr)^i < \infty.
\end{align*}
Thus, we can apply the dominated convergence theorem.

Next, we address the computation of the covariance. By the monotone convergence theorem,
\begin{align*}
R_n^{-d}\, &\E \bigl\{ \beta_{k,n}(t) \beta_{k,n}(s)\bigr\} \\
&= R_n^{-d} \sum_{i=k+2}^\infty \sum_{j, \jp \geq1} \, j \jp\, \E \Bigl\{ \sum_{\Y \subset \Pn} \gij_{n,t} (\Y, \Pn)\, \gijp_{n,s} (\Y, \Pn) \Bigr\}  \\
&\quad +R_n^{-d} \sum_{i, \ip=k+2}^\infty \sum_{j, \jp \geq1}\, j \jp\, \E \Bigl\{ \sum_{\Y \subset \Pn} \sum_{\Yp \subset \Pn, \, \Yp \neq \Y} \hspace{-5pt}\gij_{n,t} (\Y, \Pn)\, \gipjp_{n,s} (\Yp, \Pn) \Bigr\} \\
&:= A_n + B_n.
\end{align*}
The argument similar to that for deriving the limit of $R_n^{-d} \E \bigl\{ \beta_{k,n}(t) \bigr\}$ yields
$$
A_n \to \sum_{i=k+2}^\infty \sum_{j, \jp \geq1}\, j \jp\, \frac{\lambda^i}{i!}\, \muijjp_k (t,s,\lambda), \ \ \text{as } n \to \infty.
$$

As for $B_n$, note first that if $\Y$ and $\Yp$ share at least one point,
$$
\gij_{n,t}(\Y, \Pn)\, \gipjp_{n,s}(\Yp, \Pn) = 0;
$$
Therefore, it must be that $|\Y \cap \Yp| = 0$ (i.e., no common points exist between $\Y$ and $\Yp$) whenever $\Y \neq \Yp$. It then follows from the Palm theory that
\begin{equation}  \label{e:Bn}
B_n =\sum_{i, \ip = k+2}^{\infty} \sum_{j, \jp \geq 1} \, j \jp\,  R_n^{-d}  \frac{n^{i+\ip}}{i!\, \ip!}\, \E \bigl\{ \gij_{n,t}(\Y_1, \Y_{12} \cup \Pn)\, \gipjp_{n,s} (\Y_2, \Y_{12} \cup \Pn) \bigr\},
\end{equation}
where $\Y_1$ and $\Y_2$ are sets of iid points in $\bbr^d$ with density $f$, such that $|\Y_1 \cap \Y_2| = 0$, and $\Y_{12} := \Y_1 \cup \Y_2$ is independent of $\Pn$. Let $\Pnp$ be an independent copy of $\Pn$, which itself is independent of  $\Y_{12}$. Then, one more application of the Palm theory yields
\begin{align*}
R_n^{-d} &\E \bigl\{ \beta_{k,n}(t) \bigr\}\,\E \bigl\{ \beta_{k,n}(s) \bigr\} \\
&= \sum_{i, \ip = k+2}^{\infty} \sum_{j, \jp \geq 1} \, j \jp\, R_n^{-d} \frac{n^{i+\ip}}{i!\, \ip!}\, \E \bigl\{ \gij_{n,t}(\Y_1, \Y_1 \cup \Pn)\, \gipjp_{n,s} (\Y_2, \Y_2 \cup \Pnp) \bigr\}.
\end{align*}
Combining this with \eqref{e:Bn},
\begin{align*}
B_n &- R_n^{-d} \E \bigl\{ \beta_{k,n}(t) \bigr\}\,\E \bigl\{ \beta_{k,n}(s) \bigr\}  \\
&= \sum_{i, \ip = k+2}^{\infty} \sum_{j, \jp \geq 1} \, j \jp\, R_n^{-d} \frac{n^{i+\ip}}{i!\, \ip!}\, \E \bigl\{ \gij_{n,t}(\Y_1, \Y_{12} \cup \Pn)\, \gipjp_{n,s} (\Y_2, \Y_{12} \cup \Pn) \\
&\qquad \qquad \qquad - \gij_{n,t}(\Y_1, \Y_1 \cup \Pn)\, \gipjp_{n,s} (\Y_2, \Y_2 \cup \Pnp) \bigr\}.
\end{align*}
By virtue of Lemma \ref{l:conv.upper.lemma} $(iii)$, while supposing temporarily that the dominated convergence theorem is applicable, the expression on the right hand side converges to
$$
\sum_{i,  \ip = k+2}^{\infty} \sum_{j, \jp \geq 1} \, j \jp\,  \frac{\lambda^{i+\ip}}{i!\, \ip !}\, \xiijipjp_k(t,s,\lambda),
$$
and thus, $R_n^{-d} \text{Cov} \bigl\{ \beta_{k,n}(t), \beta_{k,n}(s) \bigr\} \to C_k(t,s; \infty)$, $n\to\infty$ follows, as required.

To establish a summable upper bound, we use Lemma \ref{l:conv.upper.lemma} $(iv)$ and Lemma \ref{l:geo.lemma} $(ii)$. We have that
\begin{align*}
\sum_{i,  \ip = k+2}^{\infty} &\sum_{j, \jp \geq 1} \, j \jp\, R_n^{-d} \frac{n^{i+\ip}}{i!\, \ip!}\, \Bigl| \E \bigl\{ \gij_{n,t}(\Y_1, \Y_{12} \cup \Pn)\, \gipjp_{n,s} (\Y_2, \Y_{12} \cup \Pn) \\
&\qquad \qquad \qquad - \gij_{n,t}(\Y_1, \Y_1 \cup \Pn)\, \gipjp_{n,s} (\Y_2, \Y_2 \cup \Pnp) \bigr\} \Bigr|  \\
&\leq C^* \sum_{i,  \ip = k+2}^{\infty} \frac{\bigl( \lambda(1+\delta) \bigr)^{i+\ip}}{i!\, \ip!}\, \begin{pmatrix} i \\ k+2 \end{pmatrix} \begin{pmatrix} \ip \\ k+2 \end{pmatrix} i^{i-1} (\ip)^{\ip -1} (\omega_d)^{i+\ip-1} \\
&\leq C^* \left( \sum_{i=k+2}^\infty i^{k+1} \bigl( \lambda(1+\delta) e \omega_d \bigr)^i \right)^2 < \infty.
\end{align*}
At the last inequality, we used Stirling's formula, i.e., $i! \geq (i/e)^i$ for sufficiently large $i$.
\end{proof}
\begin{lemma} \label{l:conv.upper.lemma}
Throughout the statements $(i)$ and $(ii)$ below, $\Yp$ denotes a set of iid points in $\bbr^d$ with density $f$, independent of $\Pn$.

\noindent $(i)$ For $i \geq k+2$, $j, \jp \geq 1$, and $t,s \geq 0$,
$$
R_n^{-d} n^i \E \bigl\{ \gij_{n,t}(\Yp, \Yp \cup \Pn)\, \gijp_{n,s}(\Yp, \Yp \cup \Pn) \bigr\} \to \lambda^i \muijjp_k(t,s,\lambda), \ \ \ n \to\infty.
$$
\noindent $(ii)$ There exists a positive integer $N \in \bbn_+$ such that for all $i \geq k+2$, $j, \jp \geq 1$, and $t,s\geq 0$,
\begin{align*}
R_n^{-d} &n^i \E \bigl\{ \gij_{n,t}(\Yp, \Yp \cup \Pn)\, \gijp_{n,s}(\Yp, \Yp \cup \Pn) \bigr\} \\
&\leq C^* \bigl( \lambda(1+\delta) \bigr)^i \int_{(\bbr^d)^{i-1}} \hij_t(0,\by)\, \hijp_s(0,\by) d\by, \ \ \text{for all } n \geq N,
\end{align*}
where $\delta>0$ satisfies $\lambda (1+\delta) e\omega_d < 1$.
\vspace{5pt}

Moreover, throughout $(iii)$ and $(iv)$ below, $\Y_1$ and $\Y_2$ denote sets of iid points in $\bbr^d$ with density $f$ such that $|\Y_1 \cap \Y_2| = 0$ and $\Y_{12} := \Y_1 \cup \Y_2$ is independent of $\Pn$. Let $\Pnp$ be an independent copy of $\Pn$, which is independent of $\Y_{12}$.

\noindent $(iii)$ For $i, \ip \geq k+2$, $j,\jp \geq 1$, and $t,s\geq 0$,
\begin{align*}
R_n^{-d} n^{i+\ip} \E &\bigl\{  \gij_{n,t}(\Y_1, \Y_{12} \cup \Pn)\, \gipjp_{n,s}(\Y_2, \Y_{12} \cup \Pn) \\
&- \gij_{n,t}(\Y_1, \Y_1 \cup \Pn)\, \gipjp_{n,s} (\Y_2, \Y_2 \cup \Pnp)\bigr\} \to \lambda^{i+\ip} \xiijipjp_k(t,s,\lambda), \ \ \ n\to\infty.
\end{align*}
\noindent $(iv)$ There exists a positive integer $N \in \bbn_+$ such that for all $i, \ip \geq k+2$, $j, \jp \geq 1$, and $t,s \geq 0$,
\begin{align*}
\biggl| &R_n^{-d} n^{i+\ip} \E \bigl\{  \gij_{n,t}(\Y_1, \Y_{12} \cup \Pn)\, \gipjp_{n,s}(\Y_2, \Y_{12} \cup \Pn) \\
&\qquad \qquad - \gij_{n,t}(\Y_1, \Y_1 \cup \Pn)\, \gipjp_{n,s} (\Y_2, \Y_2 \cup \Pnp) \bigr\}  \biggr| \\
&\quad \leq C^* \bigl( \lambda (1+\delta) \bigr)^{i+\ip}\int_{(\bbr^d)^{i+\ip-1}} \hspace{-10pt} \hijipjp_{t,s} (0,\by)\,\one_{\Diip(t\vee s)}(0,\by)d\by, \ \ \text{for all } n \geq N,
\end{align*}
where $\delta$ is the same positive constant as in $(ii)$.
\end{lemma}
\begin{proof}[Proof of $(i)$]
Conditioning on $\Yp$, we have that
\begin{align*}
&R_n^{-d} n^i \E \bigl\{ \gij_{n,t}(\Yp, \Yp \cup \Pn)\, \gijp_{n,s}(\Yp, \Yp \cup \Pn) \bigr\} \\
&\quad = R_n^{-d} n^i \E \biggl\{ \hij_{n,t}(\Yp)\, \hijp_{n,s}(\Yp) \P \Bigl\{ \Pn \bigl( \B (\Yp; s \vee t) \bigr) = \emptyset \, \Bigl|\,  \Yp \Bigr\} \biggr\}    \\
&\quad = R_n^{-d} n^i \int_{(\bbr^d)^i} f(\bx)\, \one \bigl\{ m(\bx) \geq R_n \bigr\} \hij_t(\bx)\, \hijp_s(\bx) \exp \bigl\{-n\int_{\B(\bx; s \vee t)} \hspace{-5pt} f(z)dz\bigr\}   d\bx.
\end{align*}
Let $J_n$ denote the last integral. Changing the variables in the same way as in \eqref{e:1st.change} and \eqref{e:polar1} yields
\begin{align}
J_n &= \bigl(nf(R_ne_1)  \bigr)^i \int_{S_{d-1}} J(\theta) d\theta \int_1^{\infty} d\rho \int_{(\bbr^d)^{i-1}} \hspace{-5pt} d\by \, \rho^{d-1} \frac{f(R_n\rho e_1)}{f(R_ne_1)}  \label{e:polar2} \\
&\qquad \qquad \qquad \times \prod_{\ell = 1}^{i-1}\, \frac{f \bigl( R_n \| \rho \theta + y_\ell /R_n \|e_1 \bigr)}{f(R_ne_1)}\, \one \bigl\{ \| \rho \theta + y_\ell /R_n \| \geq 1 \bigr\} \notag \\
&\qquad \qquad \qquad \times \hij_t(0,\by)\, \hijp_s(0,\by)\, \exp \Bigl\{ -n \int_{\B(R_n\rho \theta,R_n\rho \theta + \by; s\vee t )} \hspace{-15pt} f(z)dz \Bigr\}, \notag
\end{align}
where $S_{d-1}$ is the $(d-1)$-dimensional unit sphere  in $\bbr^d$ and $J(\theta)$ is the Jacobian. \\
By the regular variation assumption \eqref{e:RV.tail} of $f$, we have that for every $\rho \geq 1$, $\theta \in S_{d-1}$, and $y_1,\dots, y_{i-1} \in \bbr^d$,
$$
\frac{f(R_n\rho e_1)}{f(R_ne_1)}\, \prod_{\ell = 1}^{i-1}\, \frac{f \bigl( R_n \| \rho \theta + y_\ell /R_n \|e_1 \bigr)}{f(R_ne_1)} \to \rho^{-\alpha i}, \ \ \ n\to\infty.
$$
Appealing to Potter's bound as in \eqref{e:P-Potter1} and \eqref{e:P-Potter2}, for every $\rho \geq 1$, $\theta \in S_{d-1}$, and $y_1,\dots,y_{i-1} \in \bbr^d$,
\begin{align*}
n &\int_{\B(R_n\rho \theta,R_n\rho \theta + \by; s\vee t )} \hspace{-10pt} f(z)dz\\
&= nf(R_ne_1) \int_{\B(0,\by; s\vee t)} f \bigl( R_n \| \rho \theta + z /R_n \|e_1 \bigr)/f(R_ne_1) dz \\
&\to \lambda \rho^{-\alpha} (s\vee t)^d \text{vol}\bigl( \B(0,\by; 1) \bigr), \ \ \ n \to\infty.
\end{align*}

For an application of the dominated convergence theorem, we employ Potter's bound once again. First, we choose $\delta$, as in the statement of the lemma, so that $\lambda (1+\delta) e\omega_d < 1$, and then, fix $\xi \in \bigl( 0,\min \{ \alpha-d,  \delta \} \bigr)$. Then, there exists a positive integer $N_1 \in \bbn_+$, which is independent of $i$, such that
\begin{equation}  \label{e:potter1}
\frac{f(R_n\rho e_1)}{f(R_ne_1)}\, \one \{  \rho \geq 1\} \leq (1+\xi) \rho^{-\alpha +\xi}  \one \{  \rho \geq 1\}
\end{equation}
and 
\begin{equation}  \label{e:potter2}
\prod_{\ell = 1}^{i-1}\, \frac{f \bigl( R_n \| \rho \theta + y_\ell /R_n \|e_1 \bigr)}{f(R_ne_1)}\, \one \bigl\{ \| \rho \theta + y_\ell /R_n \| \geq 1 \bigr\} \leq (1+\xi)^{i-1}
\end{equation}
for all $n \geq N_1$. The integrand in \eqref{e:polar2} is now bounded above by $C^* (1+\xi)^i \rho^{d-1-\alpha+\xi}\, \hij_t(0,\by)\, \hijp_s(0,\by)$, and,
$$
\int_1^{\infty} \int_{(\bbr^d)^{i-1}} \rho^{d-1-\alpha +\xi}\, \hij_t(0,\by)\, \hijp_s(0,\by) d\by d\rho <\infty.
$$
Therefore, the dominated convergence theorem concludes that $J_n \to \lambda^i \muijjp(t,s,\lambda)$, $n\to\infty$, as required. \\
\textit{Proof of $(ii)$}:
Note first that there exists a positive integer $N_2 \in \bbn_+$ so that
\begin{equation}  \label{e:upper.lambda}
(1+\xi) nf(R_ne_1) \leq \lambda (1+\delta) \ \ \text{for all } n\geq N_2.
\end{equation}
Because of \eqref{e:potter1} and \eqref{e:potter2}, we have, for all $n \geq N := N_1 \vee N_2$,
\begin{align*}
J_n &\leq \bigl(  (1+\xi) nf(R_ne_1) \bigr)^i s_{d-1} \int_1^\infty \rho^{d-1-\alpha+\xi} d\rho \int_{(\bbr^d)^{i-1}} \hij_t(0,\by)\, \hijp_s(0,\by) d\by \\
&\leq C^* \bigl( \lambda (1+\delta) \bigr)^i \int_{(\bbr^d)^{i-1}} \hij_t(0,\by)\, \hijp_s(0,\by) d\by.
\end{align*}
\textit{Proof of $(iii)$}:
First, we write
\begin{align*}
&\E \bigl\{  \gij_{n,t}(\Y_1, \Y_{12} \cup \Pn)\, \gipjp_{n,s}(\Y_2, \Y_{12} \cup \Pn) - \gij_{n,t}(\Y_1, \Y_1 \cup \Pn)\, \gipjp_{n,s} (\Y_2, \Y_2 \cup \Pnp)\bigr\} \\
&\ = \E \bigl\{  \gij_{n,t}(\Y_1, \Y_{12} \cup \Pn)\, \gipjp_{n,s}(\Y_2, \Y_{12} \cup \Pn) - \gij_{n,t}(\Y_1, \Y_1 \cup \Pn)\, \gipjp_{n,s} (\Y_2, \Y_2 \cup \Pn)\bigr\} \\
&\ + \E \Bigl\{ \gij_{n,t}(\Y_1, \Y_1 \cup \Pn) \Bigl( \gipjp_{n,s}(\Y_2, \Y_2 \cup \Pn) - \gipjp_{n,s}(\Y_2, \Y_2 \cup \Pnp) \Bigr) \Bigr\} \\
&\ := \E \{ A_n \} + \E \{ B_n \}.
\end{align*}
Observing that
\begin{align*}
\gij_{n,t}(\Y_1, \Y_{12} \cup \Pn) &= \gij_{n,t}(\Y_1, \Y_1 \cup \Pn)\, \one \bigl\{  \, \B (\Y_1; t/2) \cap \B(\Y_2; t/2) = \emptyset \, \bigr\}, \\
\gipjp_{n,s}(\Y_2, \Y_{12} \cup \Pn) &= \gipjp_{n,s}(\Y_2, \Y_2 \cup \Pn)\, \one \bigl\{  \, \B (\Y_1; s/2) \cap \B(\Y_2; s/2) = \emptyset \, \bigr\},
\end{align*}
one can rewrite $\E \{  A_n\}$ as
$$
\E \{ A_n \} = - \E \bigl\{ \gij_{n,t}(\Y_1, \Y_1 \cup \Pn)\, \gipjp_{n,s}(\Y_2, \Y_2 \cup \Pn)\, \one_{\Diip ((t\vee s)/2)} (\Y_1, \Y_2) \bigr\}.
$$
Next, we split $\E \{B_n\}$ into two parts.
\begin{equation}  \label{e:EBn}
\E \{B_n\} = \E \Bigl\{\Delta_n \one \bigl\{ \, \B (\Y_1; t) \cap \B (\Y_2; s) =\emptyset \, \bigr\} \Bigr\} + \E \bigl\{ \Delta_n \one_{\Diip (t,s)} (\Y_1, \Y_2)\bigr\},
\end{equation}
where
$$
\Delta_n = \gij_{n,t}(\Y_1, \Y_1 \cup \Pn) \Bigl( \gipjp_{n,s}(\Y_2, \Y_2 \cup \Pn) - \gipjp_{n,s}(\Y_2, \Y_2 \cup \Pnp) \Bigr).
$$
By the spacial independence of the Poisson point process, the first term on the right hand side of \eqref{e:EBn} equals zero.
Rearranging the terms in $\E\{A_n\}$ and $\E\{B_n\}$, we obtain
\begin{align*}
&\E \bigl\{  \gij_{n,t}(\Y_1, \Y_{12} \cup \Pn)\, \gipjp_{n,s}(\Y_2, \Y_{12} \cup \Pn) - \gij_{n,t}(\Y_1, \Y_1 \cup \Pn)\, \gipjp_{n,s} (\Y_2, \Y_2 \cup \Pnp) \bigr\} \\
&\ =\E\{C_n\} -\E\{D_n\},
\end{align*}
where
\begin{align*}
C_n &= \gij_{n,t}(\Y_1, \Y_1 \cup \Pn)\, \gipjp_{n,s}(\Y_2, \Y_2 \cup \Pn) \\
&\qquad \qquad \times \Bigl( \one_{\Diip(t,s)} (\Y_1,\Y_2) - \one_{\Diip((t\vee s)/2)} (\Y_1,\Y_2) \Bigr), \\
D_n &= \gij_{n,t}(\Y_1, \Y_1 \cup \Pn)\, \gipjp_{n,s}(\Y_2, \Y_2 \cup \Pnp)\, \one_{\Diip(t,s)} (\Y_1,\Y_2).
\end{align*}
Conditioning on $\Y_{12}$, we have
\begin{align*}
\E \{ C_n\} &= \int_{(\bbr^d)^i}\int_{(\bbr^d)^{\ip}} f(\bx_1)\,f(\bx_2)\, \one \bigl\{ m(\bx_1,\bx_2) \geq R_n \bigr\} \hij_t(\bx_1)\,\hipjp_s(\bx_2) \\
&\qquad \qquad \quad \times \Bigl( \one_{\Diip(t,s)}(\bx_1,\bx_2) - \one_{\Diip((t\vee s)/2)}(\bx_1,\bx_2) \Bigr) \\
&\qquad \qquad \quad \times  \exp \Bigl\{ -n\int_{\B(\bx_1; t) \cup \B(\bx_2; s)} \hspace{-10pt}f(z)dz \Bigr\}\, d\bx_2 d\bx_1.
\end{align*}
Proceeding as in the proof of $(i)$, while suitably applying Potter's bound, we can obtain, as $n\to\infty$,
\begin{align*}
R_n^{-d} n^{i + \ip} \E \{ C_n\}  &\to \lambda^{i+\ip}  s_{d-1} \int_1^\infty \rho^{d-1-\alpha(i+\ip)} \int_{(\bbr^d)^{i+\ip-1}} \hijipjp_{t,s}(0,\by) \\
&\qquad \qquad \qquad \times \Bigl( \one_{\Diip(t,s)}(0,\by) - \one_{\Diip((t\vee s)/2)}(0,\by) \Bigr) \\
&\qquad \qquad \qquad \times e^{-\lambda \rho^{-\alpha} \text{vol} \bigl( \B(0,y_1,\dots y_{i-1}; t) \cup \B (y_i, \dots, y_{i+\ip-1}; s) \bigr)} d\by d\rho.
\end{align*}
Similarly, we have
\begin{align*}
R_n^{-d} n^{i + \ip} \E \{ D_n\}  &\to  \lambda^{i+\ip} s_{d-1} \int_1^\infty \rho^{d-1-\alpha(i+\ip)} \int_{(\bbr^d)^{i+\ip-1}} \hspace{-10pt} \hijipjp_{t,s}(0,\by)\, \one_{\Diip(t,s)}(0,\by) \\
&\qquad \qquad \qquad \times e^{-\lambda \rho^{-\alpha} \bigl[ \text{vol}\bigl(\B(0,y_1,\dots y_{i-1}; t)\bigr) +  \text{vol}\bigl(\B(y_i,\dots y_{i+\ip-1}; s)\bigr) \bigr]  } d\by d\rho,
\end{align*}
and, therefore,
$$
R_n^{-d} n^{i + \ip} \bigl( \E \{ C_n\}-\E\{D_n \}\bigr) \to \lambda^{i+\ip} \xiijipjp_k(t,s,\lambda), \ \ \ n\to\infty.
$$
\textit{Proof of $(iv)$}: Note first that
\begin{align*}
\biggl| R_n^{-d} &n^{i+\ip} \E \bigl\{  \gij_{n,t}(\Y_1, \Y_{12} \cup \Pn)\, \gipjp_{n,s}(\Y_2, \Y_{12} \cup \Pn) \\
&\qquad \qquad - \gij_{n,t}(\Y_1, \Y_1 \cup \Pn)\, \gipjp_{n,s} (\Y_2, \Y_2 \cup \Pnp) \bigr\}  \biggr| \\
&= \Bigl| R_n^{-d} n^{i+\ip} \bigl( \E \{C_n\} - \E\{D_n \}\bigr) \Bigr|  \\
&\leq 2R_n^{-d} n^{i+\ip} \E \bigl\{ \hij_{n,t}(\Y_1)\, \hipjp_{n,s}(\Y_2)\, \one_{\Diip(t\vee s)}(\Y_1, \Y_2) \bigr\}.
\end{align*}
Changing the variables in the same manner as in $(i)$, the last expression above equals
\begin{align*}
2 \bigl( nf(R_ne_1) \bigr)^{i+\ip} &\int_{S_{d-1}} J(\theta)d\theta \int_1^{\infty} d\rho \int_{(\bbr^d)^{i+\ip-1}} \hspace{-10pt} d\by\, \rho^{d-1} \frac{f(R_n\rho e_1)}{f(R_ne_1)} \\
&\times \prod_{\ell = 1}^{i+\ip-1} \frac{f \bigl(R_n \| \rho \theta + y_\ell/R_n \|e_1 \bigr)}{f(R_ne_1)}\, \one \bigl\{ \| \rho \theta + y_\ell /R_n  \| \geq 1 \bigr\} \\
&\times  \hijipjp_{t,s}(0,\by)\, \one_{\Diip(t\vee s)}(0,\by)
\end{align*}
Using the upper bound \eqref{e:potter1} and
$$
\prod_{\ell = 1}^{i+\ip-1}\, \frac{f \bigl( R_n \| \rho \theta + y_\ell /R_n \|e_1 \bigr)}{f(R_ne_1)}\, \one \bigl\{ \| \rho \theta + y_\ell /R_n \| \geq 1 \bigr\} \leq (1+\xi)^{i+\ip-1},
$$
and applying \eqref{e:upper.lambda}, we can complete the proof.
\end{proof}
\begin{lemma} \label{l:geo.lemma}
Fix a positive constant $L>0$. \\
\noindent $(i)$ For $i \geq k+2$, $0 \leq t,s \leq L$,
\begin{align*}
\sum_{j, \jp \geq1} j \jp \int_{(\bbr^d)^{i-1}} \hij_t(0,\by)\, \hijp_s(0,\by) d\by \leq \begin{pmatrix} i \\ k+2 \end{pmatrix}^2 i^{i-2} (L^d \omega_d)^{i-1},
\end{align*}
where $\omega_d$ is a volume of the unit ball in $\bbr^d$.
\vspace{5pt}

\noindent $(ii)$ For $i, \ip \geq k+2$, $0 \leq t,s \leq L$,
\begin{align*}
&\sum_{j, \jp \geq1} j \jp \int_{(\bbr^d)^{i+\ip-1}} \hspace{-10pt} \hijipjp_{t,s} (0,\by)\,  \one_{\Diip(t\vee s)}(0,\by) d\by \\
&\quad \leq 2^{d} \begin{pmatrix} i \\ k+2 \end{pmatrix} \begin{pmatrix} \ip \\ k+2 \end{pmatrix} i^{i-1} (\ip)^{\ip-1} (L^d \omega_d)^{i+\ip-1}.
\end{align*}
\end{lemma}
\begin{proof}[Proof of $(i)$]
Since every connected component built on a set of $i$ points can contribute to the $k$-th Betti number at most $\begin{pmatrix} i \\ k+2 \end{pmatrix}$ times, we have that
\begin{align*}
&\sum_{j, \jp \geq1} j \jp \int_{(\bbr^d)^{i-1}} \hij_t(0,\by)\, \hijp_s(0,\by) d\by \\
&\quad \leq \begin{pmatrix} i \\ k+2 \end{pmatrix}^2 \sum_{j,\jp\geq1} \int_{(\bbr^d)^{i-1}} \hij_t(0,\by)\, \hijp_s(0,\by) d\by  \\
&\quad \leq \begin{pmatrix} i \\ k+2 \end{pmatrix}^2 L^{d(i-1)} \int_{(\bbr^d)^{i-1}} \hspace{-5pt}\one \bigl\{ \, \check{C}(0,\by; 1) \text{ is connected} \, \bigr\} d\by.
\end{align*}
It is well known that there exist $i^{i-2}$ spanning trees on a set of $i$ points, and thus,
$$
 \int_{(\bbr^d)^{i-1}} \hspace{-5pt}\one \bigl\{ \, \check{C}(0,\by; 1) \text{ is connected} \, \bigr\} d\by \leq i^{i-2} (\omega_d)^{i-1}.
$$
Now, the claim is proved. \\
\textit{Proof of $(ii)$}:
\begin{align*}
&\sum_{j, \jp \geq1} j \jp \int_{(\bbr^d)^{i+\ip-1}} \hspace{-10pt} \hijipjp_{t,s} (0,\by)\,  \one_{\Diip(t\vee s)}(0,\by) d\by \\
&\quad \leq \begin{pmatrix} i \\ k+2 \end{pmatrix} \begin{pmatrix} \ip \\ k+2 \end{pmatrix} L^{d(i+\ip-1)}\int_{(\bbr^d)^{i+\ip-1}} \hspace{-5pt} \one \bigl\{ \, \check{C}(0,y_1,\dots,y_{i-1}; 1) \text{ is connected},  \\
&\qquad \qquad \qquad \quad \check{C}(y_i, \dots, y_{i+\ip-1}; 1) \text{ is connected}, \ \check{C}(0,\by; 2) \text{ is connected} \bigr\} d\by.
\end{align*}
If $\check{C}(0,y_1,\dots,y_{i-1}; 1)$ is connected, there exist $i^{i-2}$ spanning trees constructed from $\{ 0,y_1,\dots,y_{i-1} \}$. Similarly, there are $(\ip)^{\ip -2}$ spanning trees built on the points $\{ y_i, \dots, y_{i+\ip-1}\}$ whenever $\check{C}(y_i,\dots,y_{i+\ip-1}; 1)$ is connected. In addition, if $\check{C}(0,\by; 2)$ is connected, two sets of points $\{ 0,y_1,\dots,y_{i-1}\}$ and $\{ y_i,\dots,y_{i+\ip-1} \}$ must be at a distance of at most $2$, implying that $\| y_p - y_q \| \leq 2$ for some $p \in \{ 0,\dots,i-1 \}$ and $q\in \{ i,\dots,i+\ip-1 \}$ (take $y_0 \equiv 0$). Therefore,
\begin{align*}
&\int_{(\bbr^d)^{i+\ip-1}} \hspace{-5pt} \one \bigl\{ \, \check{C}(0,y_1,\dots,y_{i-1}; 1) \text{ is connected}, \ \check{C}(y_i, \dots, y_{i+\ip-1}; 1) \text{ is connected}, \\
&\quad \qquad \qquad \qquad \qquad  \check{C}(0,\by; 2) \text{ is connected} \bigr\} d\by  \\
&\quad \leq i^{i-2} (\ip)^{\ip-2} i \ip (\omega_d)^{i+\ip-2} 2^d \omega_d = 2^d i^{i-1} (\ip)^{\ip-1} (\omega_d)^{i+\ip-1}.
\end{align*}
\end{proof}

Subsequently, we establish the FCLT for the truncated Betti number \eqref{e:truncated.betti}, for which, as its limit, we need to define a ``truncated" limiting Gaussian process. For $M \geq k+2$, we define
$$
Z_k^{(M)}(t) := \sum_{i=k+2}^M \sum_{j\geq1} j Z_k^{(i,j)}(t),  \ \ \ t\geq0.
$$
It is worthwhile noting that there is no need to restrict the range of $\lambda$ as in \eqref{e:Rn.regime3}. Further, we do not need to restrict the domain of functions in the space $C$.
\begin{lemma}  \label{l:giant.truncated.clt}
Suppose that
$$
nf(R_ne_1) \to \lambda \in (0,\infty), \ \ \ n\to\infty.
$$
Then, for every $M\geq k+2$,
$$
R_n^{-d/2} \int_0^t \Bigl( \beta_{k,n}^{(M)}(s) - \E \bigl\{ \beta_{k,n}^{(M)}(s) \bigr\} \Bigr)ds \Rightarrow \int_0^t Z_k^{(M)}(s)ds \ \ \text{in } C[0,\infty).
$$
\end{lemma}
\begin{proof}
Our proof is closely related to that in Theorem 3.9 in \cite{penrose:2003}.
To prove finite-dimensional weak convergence, we apply the Cram\'er-Wold device, for which we need to establish the central limit theorem for
$$
R_n^{-d/2} \sum_{p=1}^m a_p \int_0^{t_p} \Bigl( \beta_{k,n}^{(M)}(s) - \E \bigl\{ \beta_{k,n}^{(M)}(s) \bigr\} \Bigr) ds
$$
for every $a_1,\dots,a_m \in \bbr$, $0 \leq t_1 < \dots < t_m <\infty$, and $m\geq 1$. \\
We first decompose this term into two parts in the following manner. For $K\geq1$, we write
\begin{align*}
\sum_{p=1}^m a_p \int_0^{t_p} \beta_{k,n}^{(M)}(s) ds &= \sum_{p=1}^m a_p \int_0^{t_p} \beta_{k,n}^{(M)}(s; K) ds  \\
&+ \sum_{p=1}^m a_p \int_0^{t_p} \sum_{i=k+2}^M \sum_{j\geq1} j \sum_{\Y \subset \Pn} \gij_{n,s}(\Y, \Pn)\, \\
& \qquad \qquad \times \one \bigl\{ \text{Max}(\Y) \in \text{Ann}(KR_n,\infty) \bigr\}ds  \\
&:= T_n^{(M)}(K) + U_n^{(M)}(K).
\end{align*}
Define
$$
\gamma^{(M)}(K) = \sum_{p=1}^m\sum_{q=1}^m a_p a_q \int_0^{t_p}\int_0^{t_q} C_k^{(M)}(u,v; K)\, du dv,
$$
where $C_k^{(M)} (u,v; K)$ is a truncated version of $C_k (u,v; K)$ given by
\begin{align*}
C_k^{(M)} (u,v; K) &:= \sum_{i=k+2}^M \sum_{j, \jp \geq1}\, j \jp\, \frac{\lambda^i}{i!}\, \muijjp_k (u,v,\lambda; K)  \\
&\qquad +\sum_{i, \ip = k+2}^M \sum_{j, \jp \geq1} \, j \jp\, \frac{\lambda^{i + \ip}}{i! \, \ip!} \, \xiijipjp_k(u,v,\lambda; K).
\end{align*}
Moreover, $\gamma^{(M)} := \lim_{K\to\infty}\gamma^{(M)}(K)$. It then follows from Lemma \ref{l:giant.cov} that
$$
\gamma^{(M)}(K) = \lim_{n\to\infty} R_n^{-d}\text{Var} \{ T_n^{(M)}(K) \}.
$$
For the required finite-dimensional weak convergence, we need to show that for every $M \geq k+2$,
$$
R_n^{-d/2} \sum_{p=1}^m a_p \int_0^{t_p} \Bigl( \beta_{k,n}^{(M)}(s) - \E \bigl\{ \beta_{k,n}^{(M)}(s) \bigr\} \Bigr) ds \Rightarrow N(0,\gamma^{(M)}), \ \ \ n\to\infty.
$$
By the standard approximation argument given on p. $\hspace{-5pt}$ 64 in \cite{penrose:2003}, it suffices to show that for every $K\geq1$,
$$
R_n^{-d/2} \bigl( T_n^{(M)}(K) - \E \{ T_n^{(M)}(K) \} \bigr) \Rightarrow N\bigl(0,\gamma^{(M)}(K)\bigr),  \ \ \ n\to \infty;
$$
equivalently, as $n\to\infty$,
$$
\frac{T_n^{(M)}(K) - \E \{ T_n^{(M)}(K) \}}{\sqrt{\text{Var}\{ T_n^{(M)}(K) \}}} \Rightarrow N(0,1) \ \ \text{for every } K\geq1.
$$

Let $(Q_\ell, \, \ell \in \bbn)$ be unit cubes covering $\bbr^d$. Let
$$
V_n := \bigl\{ \ell \in \bbn: Q_\ell \cap \text{Ann}(R_n,KR_n) \neq \emptyset \, \bigr\}.
$$
Then, we see that $| V_n  | \leq C^* R_n^d$. \\
Subsequently, we partition $T_n^{(M)}(K)$ as follows.
\begin{align*}
T_n^{(M)}(K) &= \sum_{\ell \in V_n} \sum_{p=1}^m a_p \int_0^{t_p} \sum_{i=k+2}^M \sum_{j\geq1} j \sum_{\Y \subset \Pn} \gij_{n,s}(\Y, \Pn)\,\\
&\qquad \qquad \qquad \times  \one \bigl\{ \text{Max}(\Y) \in \text{Ann}(R_n,KR_n) \cap Q_\ell \bigr\}ds \\
&:= \sum_{\ell \in V_n} \xi_{\ell,n}.
\end{align*}
We define a relation $\sim$ on a vertex set $V_n$ by $\ell \sim \ellp$ if and only if the distance between $Q_\ell$ and $Q_\ellp$ is less than $2Mt_m$. In this case, $(V_n, \sim)$ constitutes a \textit{dependency graph}, that is, for any two vertex sets $I_1, I_2 \subset V_n$ with no edges connecting them, $(\xi_{\ell, n}, \, \ell \in I_1)$ and $(\xi_{\ellp, n}, \, \ellp \in I_2)$ are independent.
By virtue of Stein's method for normal approximation (see Theorem 2.4 in \cite{penrose:2003}), the proof will be complete, provided that for $p=3,4$,
$$
R_n^d\, \max_{\ell \in V_n}\, \frac{E \bigl| \xi_{\ell,n} - \E\{ \xi_{\ell,n} \} \bigr|^p}{\bigl( \text{Var}\{ T_n^{(M)}(K) \} \bigr)^{p/2}} \to 0, \ \ \ n\to \infty.
$$
For $\ell \in V_n$, we denote by $Z_{\ell,n}$ the number of points in $\Pn$ lying in
$$
\text{Tube}(Q_\ell; Mt_m) := \bigl\{ x\in\bbr^d: \inf_{y \in Q_\ell}\|x-y \| \leq Mt_m \bigr\}.
$$
Clearly, $Z_{\ell,n}$ possesses a Poisson law with mean $n\int_{\text{Tube}(Q_\ell; Mt_m)}f(z)dz$. Using Potter's bound, we see that $Z_{\ell,n}$ is stochastically dominated by another Poisson random variable with a constant mean $C^*$. \\
Observe that
$$
|\xi_{\ell,n} |\leq \sum_{p=1}^m |a_p| t_m \sum_{i=k+2}^M \begin{pmatrix} i \\ k+2 \end{pmatrix}^2 \begin{pmatrix} z_{\ell,n} \\ i \end{pmatrix},
$$
and, accordingly, we have
$$
\max_{\ell \in V_n} \E \bigl| \xi_{\ell,n} - \E \{ \xi_{\ell,n} \} \bigr|^p \leq C^* \ \ \text{for } p=3,4.
$$
Therefore, for $p=3,4$,
$$
R_n^d\, \max_{\ell \in V_n}\, \frac{E \bigl| \xi_{\ell,n} - \E\{ \xi_{\ell,n} \} \bigr|^p}{\bigl( \text{Var}\{ T_n^{(M)}(K) \} \bigr)^{p/2}} \leq C^* R_n^d (R_n^d)^{-p/2} \to 0, \ \ \ n\to\infty,
$$
which completes the proof of the finite-dimensional weak convergence.

Next, we turn to verifying the tightness of
$$
X_n(t) := R_n^{-d/2} \int_0^t \Bigl( \beta_{k,n}^{(M)}(s) - \E \bigl\{ \beta_{k,n}^{(M)}(s) \bigr\} \Bigr)ds, \ \ t \geq 0,
$$
in the space $C[0,\infty)$. According to Theorem 12.3 in \cite{billingsley:1968}, we only have to show that, for any $L>0$, there exists $B>0$ such that
$$
\E \Bigl\{\bigl( X_n(T) - X_n(S)\bigr)^2 \Bigr\} \leq B (T-S)^2
$$
for all $0\leq S \leq T\leq L$ and $n\geq1$.

We see that
\begin{align*}
\E &\Bigl\{\bigl( X_n(T) - X_n(S)\bigr)^2 \Bigr\}  \\
&=R_n^{-d} \int_S^T\int_S^T \text{Cov} \bigl\{ \beta_{k,n}^{(M)}(t), \beta_{k,n}^{(M)}(s) \bigr\} ds dt \\
&= \int_S^T\int_S^T \sum_{i = k+2}^M \sum_{j, \jp \geq 1} \, j \jp\,  R_n^{-d}  \frac{n^{i}}{i!}\, \E \bigl\{ \gij_{n,t}(\Yp, \Yp \cup \Pn)\, \gijp_{n,s} (\Yp, \Yp \cup \Pn) \bigr\} ds dt \\
&\quad + \int_S^T\int_S^T \sum_{i, \ip = k+2}^M \sum_{j, \jp \geq1} \, j \jp\,  R_n^{-d}  \frac{n^{i+\ip}}{i!\, \ip!}\, \E \bigl\{ \gij_{n,t}(\Y_1, \Y_{12} \cup \Pn)\, \gipjp_{n,s} (\Y_2, \Y_{12} \cup \Pn) \\
&\qquad \qquad \qquad \qquad \qquad \qquad \qquad  - \gij_{n,t}(\Y_1, \Y_1 \cup \Pn)\, \gipjp_{n,s} (\Y_2, \Y_2 \cup \Pnp) \bigr\}dsdt
\end{align*}
($\Yp$ and $\Pnp$ are defined in the statement of Lemma \ref{l:conv.upper.lemma}). \\
Combining Lemma \ref{l:conv.upper.lemma} $(ii)$, $(iv)$ and Lemma \ref{l:geo.lemma} $(i)$, $(ii)$, the integrands in the last expression can be bounded above by a positive and finite constant, which does not depend on $s$, $t$, and $n$. We now conclude that
$$
\E \Bigl\{\bigl( X_n(T) - X_n(S)\bigr)^2 \Bigr\} \leq C^* (T-S)^2,
$$
and, thus, the tightness follows.
\end{proof}
\begin{proof}[Proof of Theorem \ref{t:giant.fclt}]
By Lemma \ref{l:giant.truncated.clt} and Theorem 3.2 in \cite{billingsley:1999}, it suffices to verify that for every $\epsilon > 0$,
\begin{align}
\lim_{M\to\infty} \limsup_{n\to\infty}\, \P \biggl\{\, \sup_{0\leq t \leq 1}\, \Bigl| \, \int_0^t &\Bigl( \beta_{k,n}(s)-\beta_{k,n}^{(M)}(s) \label{e:theorem3.2-1}  \\
&- \E \bigl\{ \beta_{k,n}(s)-\beta_{k,n}^{(M)}(s) \bigr\} \Bigr)ds\, \Bigr| >\epsilon R_n^{d/2} \biggr\} = 0, \notag
\end{align}
and
\begin{equation} \label{e:theorem3.2-2}
\lim_{M\to\infty}\P \biggl\{\, \sup_{0\leq t \leq 1}\, \Bigl| \, \int_0^t \Bigl(Z_k(s) - Z_k^{(M)}(s)  \Bigr) ds\,  \Bigr| > \epsilon \biggr\} = 0.
\end{equation}
By Chebyshev's inequality, \eqref{e:theorem3.2-1} immediately follows, provided that
\begin{align*}
\lim_{M\to\infty} \limsup_{n\to\infty}\, R_n^{-d} \E \biggl\{\, \sup_{0\leq t \leq 1}\, \Bigl| \, \int_0^t &\Bigl( \beta_{k,n}(s)-\beta_{k,n}^{(M)}(s) \\
&- \E \bigl\{ \beta_{k,n}(s)-\beta_{k,n}^{(M)}(s) \bigr\} \Bigr)ds\, \Bigr|^2 \biggr\} = 0
\end{align*}
By Cauchy-Schwarz inequality, we only have to show that
$$
\lim_{M\to\infty} \limsup_{n\to\infty}\, \int_0^1 \Bigl( R_n^{-d} \text{Var} \bigl\{ \beta_{k,n}(t)- \beta_{k,n}^{(M)}(t) \bigr\} \Bigr)^{1/2} ds = 0.
$$
One can decompose the integrand as follows.
\begin{align*}
&R_n^{-d} \text{Var} \bigl\{ \beta_{k,n}(t)- \beta_{k,n}^{(M)}(t) \bigr\}  \\
&\quad = \sum_{i=M+1}^\infty \sum_{j\geq1} j^2 R_n^{-d}\, \frac{n^i}{i!}\, \E \bigl\{ \gij_{n,t}(\Yp, \Yp \cup \Pn) \bigr\}  \\
&\qquad + \sum_{i,\ip=M+1}^\infty \sum_{j,\jp \geq1} j\jp\, R_n^{-d}\, \frac{n^{i+\ip}}{i!\,\ip !}\, \E\bigl\{ \gij_{n,t}(\Y_1, \Y_{12} \cup \Pn)\, \gipjp_{n,t}(\Y_2,\Y_{12} \cup \Pn) \\
&\qquad \qquad \qquad \qquad \qquad \qquad \qquad \quad - \gij_{n,t}(\Y_1, \Y_{1} \cup \Pn)\, \gipjp_{n,t}(\Y_2,\Y_{2} \cup \Pnp) \bigr\}
\end{align*}
($\Yp$ and $\Pnp$ are defined in the statement of Lemma \ref{l:conv.upper.lemma}). \\
Combining Lemma \ref{l:conv.upper.lemma} $(ii)$, $(iv)$ and Lemma \ref{l:geo.lemma} $(i)$, $(ii)$ proves that this is bounded by
\begin{align*}
C^* &\sum_{i=M+1}^\infty \frac{\bigl( \lambda (1+\delta) \bigr)^i}{i!}\, \begin{pmatrix} i \\ k+2 \end{pmatrix}^2 i^{i-2} (\omega_d)^{i-1}  \\
&+ C^* \sum_{i, \ip =M+1}^\infty \frac{\bigl( \lambda (1+\delta) \bigr)^{i+\ip}}{i!\, \ip !}\, \begin{pmatrix} i \\ k+2 \end{pmatrix} \begin{pmatrix} \ip \\ k+2 \end{pmatrix} i^{i-1} (\ip)^{\ip-1} (\omega_d)^{i+\ip-1} \\
&\leq C^* \sum_{i=M+1}^\infty i^{2k+2} \bigl( \lambda (1+\delta) e\omega_d \bigr)^i  + C^* \left( \sum_{i=M+1}^\infty i^{k+1} \bigl( \lambda (1+\delta) e\omega_d \bigr)^i \right)^2.
\end{align*}
Since $0<\lambda (1+\delta) e\omega_d < 1$, the claim has been proved.
Since the proof of \eqref{e:theorem3.2-2} is almost the same as that of \eqref{e:theorem3.2-1}, we omit it.
\end{proof}

\subsection{Proof of Theorem \ref{t:sparse.fclt}}  \label{s:proof.second.regime}

The proof of Theorem \ref{t:sparse.fclt} somewhat parallels that of Theorem \ref{t:giant.fclt}, for which we need to recall the notations of several indicator functions and variants of the Betti numbers defined at the beginning of Section \ref{s:proof.third.regime}. As in Lemma \ref{l:giant.cov}, we begin with computing the asymptotic mean and covariance of the scaled $k$-th Betti numbers. In the following, let $\rho_n := n^{k+2}R_n^d f(R_ne_1)^{k+2}$.
\begin{lemma}  \label{l:sparse.cov}
For every $t, s \leq 0$ and $1\leq K\leq \infty$, we have, as $n\to \infty$,
$$
\rho_n^{-1}\, \E \bigl\{ \beta_{k,n}(t; K) \bigr\} \to \mukt_k (t,t,0; K) / (k+2)! \in (0,\infty),
$$
and
\begin{align*}
\rho_n^{-1}\, \text{Cov} \bigl\{ \beta_{k,n}(t; K), \beta_{k,n}(s; K) \bigr\} \to \mukt_k (t,s,0; K) / (k+2)! \in (0,\infty),
\end{align*}
where the definition of the limit is given in \eqref{e:def.muijjp1}.
\end{lemma}

Recall that, in the last subsection, Lemmas \ref{l:conv.upper.lemma} and \ref{l:geo.lemma} play a crucial role in proving Lemma \ref{l:giant.cov}. In the present subsection, however, one needs to replace Lemma \ref{l:conv.upper.lemma} with Lemma \ref{l:sparse.conv.upper.lemma} below in order to show Lemma \ref{l:sparse.cov}. Since the proof of Lemma \ref{l:sparse.conv.upper.lemma} is analogous to that of Lemma \ref{l:conv.upper.lemma}, we omit the proof.
\begin{lemma} \label{l:sparse.conv.upper.lemma}
Throughout the statements $(i)$ and $(ii)$ below, $\Yp$ denotes a set of iid points in $\bbr^d$ with density $f$, independent of $\Pn$.

\noindent $(i)$ For $t,s \geq 0$, we have, as $n\to\infty$,
$$
\rho_n^{-1} n^{k+2} \E \bigl\{ g^{(k+2,1)}_{n,t}(\Yp, \Yp \cup \Pn)\, g^{(k+2,1)}_{n,s}(\Yp, \Yp \cup \Pn) \bigr\} \to \mukt_k(t,s,0).
$$
\noindent $(ii)$ There exists a positive integer $N \in \bbn_+$ such that for all $i \geq k+2$, $j, \jp \geq 1$, and $t,s\geq 0$,
\begin{align*}
\rho_n^{-1} n^i &\E \bigl\{ \gij_{n,t}(\Yp, \Yp \cup \Pn)\, \gijp_{n,s}(\Yp, \Yp \cup \Pn) \bigr\} \\
&\leq C^* \bigl( 2nf(R_ne_1) \bigr)^{i-(k+2)} \int_{(\bbr^d)^{i-1}} \hij_t(0,\by)\, \hijp_s(0,\by) d\by
\end{align*}
for all $n \geq N$.
\vspace{5pt}

Moreover, $\Y_1$ and $\Y_2$ denote sets of iid points in $\bbr^d$ with density $f$ such that $|\Y_1 \cap \Y_2| = 0$ and $\Y_{12} := \Y_1 \cup \Y_2$ is independent of $\Pn$. Let $\Pnp$ be an independent copy of $\Pn$, which is independent of $\Y_{12}$.

\noindent $(iii)$ There exists a positive integer $N \in \bbn_+$ such that for all $i, \ip \geq k+2$, $j, \jp \geq 1$, and $t,s \geq 0$,
\begin{align*}
\biggl| &\rho_n^{-1} n^{i+\ip} \E \bigl\{  \gij_{n,t}(\Y_1, \Y_{12} \cup \Pn)\, \gipjp_{n,s}(\Y_2, \Y_{12} \cup \Pn) \\
&\qquad \qquad - \gij_{n,t}(\Y_1, \Y_1 \cup \Pn)\, \gipjp_{n,s} (\Y_2, \Y_2 \cup \Pnp) \bigr\}  \biggr| \\
&\quad \leq C^* \bigl( 2nf(R_ne_1) \bigr)^{i+\ip-(k+2)} \int_{(\bbr^d)^{i+\ip-1}} \hspace{-10pt} \hijipjp_{t,s} (0,\by)\,\one_{\Diip(t\vee s)}(0,\by)d\by
\end{align*}
for all $n \geq N$.
\end{lemma}

\begin{proof}[Proof of Lemma \ref{l:sparse.cov}]
As in the proof of Lemma \ref{l:giant.cov}, we may prove only the case $K=\infty$. Moreover, we compute only the limit of scaled covariance by $\rho_n$. Proceeding as in the proof of Lemma \ref{l:giant.cov}, one can write
\begin{align*}
\rho_n^{-1}\, &\text{Cov} \bigl\{ \beta_{k,n}(t), \beta_{k,n}(s) \bigr\}  \\
&= \sum_{i=k+2}^\infty \sum_{j, \jp \geq 1} j \jp\, \rho_n^{-1}\, \frac{ n^i}{i!}\, \E \bigl\{ \gij_{n,t}(\Yp, \Yp \cup \Pn)\, \gijp_{n,s}(\Yp, \Yp \cup \Pn) \bigr\} \\
&\qquad + \sum_{i, \ip=k+2}^\infty \sum_{j, \jp \geq 1} j \jp\, \rho_n^{-1}\, \frac{ n^{i+\ip}}{i!\, \ip !}\, \E \bigl\{  \gij_{n,t}(\Y_1, \Y_{12} \cup \Pn)\, \gipjp_{n,s}(\Y_2, \Y_{12} \cup \Pn)  \\
&\qquad \qquad \qquad \qquad \qquad \qquad \qquad  - \gij_{n,t}(\Y_1, \Y_1 \cup \Pn)\, \gipjp_{n,s} (\Y_2, \Y_2 \cup \Pnp) \bigr\}.
\end{align*}
By Lemma \ref{l:sparse.conv.upper.lemma} $(i) - (iii)$, it now suffices to show that, as $n\to \infty$,
$$
A_n := \sum_{i=k+3}^\infty \sum_{j, \jp \geq 1} j \jp\, \frac{\bigl(  2nf(R_ne_1) \bigr)^{i-(k+2)}}{i!}\, \int_{(\bbr^d)^{i-1}} \hij_t(0,\by)\, \hijp_s(0,\by) d\by \to 0,
$$
and
\begin{align*}
B_n := \sum_{i, \ip =k+2}^\infty \sum_{j, \jp \geq 1} j \jp\, &\frac{\bigl(  2nf(R_ne_1) \bigr)^{i+\ip-(k+2)}}{i!\, \ip !}\, \\
&\times \int_{(\bbr^d)^{i+\ip-1}} \hspace{-10pt} \hijipjp_{t,s} (0,\by)\,\one_{\Diip(t\vee s)}(0,\by)d\by
\to 0.
\end{align*}
It follows from Lemma \ref{l:geo.lemma} $(i)$ that
\begin{align*}
A_n &\leq \sum_{i=k+3}^\infty \frac{\bigl(  2nf(R_ne_1) \bigr)^{i-(k+2)}}{i!}\, \begin{pmatrix} i \\k+2 \end{pmatrix}^2 i^{i-2}
 \bigl( (t\vee s)^d \omega_d \bigr)^{i-1} \\
&\leq C^* \sum_{i=k+3}^\infty i^{2k+2} \bigl(  2nf(R_ne_1) (t\vee s)^d e\omega_d \bigr)^{i-(k+2)}  \\
&\to 0 \ \ \text{as } n\to\infty,
\end{align*}
where the last convergence is obtained by $nf(R_ne_1) \to 0$, $n\to\infty$.  \\
Similarly, by Lemma \ref{l:geo.lemma} $(ii)$,
\begin{align*}
B_n &\leq \sum_{i, \ip =k+2}^\infty \frac{\bigl(  2nf(R_ne_1) \bigr)^{i+\ip-(k+2)}}{i!\, \ip !}\,  2^d \begin{pmatrix} i \\ k+2 \end{pmatrix} \begin{pmatrix} \ip \\ k+2 \end{pmatrix} i^{i-1} (\ip)^{\ip -1} \bigl( (t\vee s)^d \omega_d \bigr)^{i+\ip-1} \\
&\leq C^* \sum_{i, \ip =k+2}^\infty i^{k+1} (\ip)^{k+1} \bigl(  2nf(R_ne_1) (t\vee s)^d e\omega_d \bigr)^{i+\ip-(k+2)}  \\
&\to 0 \ \ \text{as } n\to \infty.
\end{align*}
\end{proof}
The next lemma claims the FCLT for the integral process associated with the truncated $k$-th Betti number  \eqref{e:truncated.betti}. The proof is almost the same as that of Lemma \ref{l:giant.truncated.clt}, and therefore, we do not state it here. It is then straightforward to complete the proof of Theorem \ref{t:sparse.fclt} by combining Lemma \ref{l:sparse.truncated.clt} and Theorem 3.2 in \cite{billingsley:1999}, as in the last subsection.
\begin{lemma}  \label{l:sparse.truncated.clt}
For every $M\geq k+2$, we have, as $n\to\infty$,
$$
\rho_n^{-1/2} \int_0^t \Bigl( \beta_{k,n}^{(M)}(s) - \E \bigl\{ \beta_{k,n}^{(M)}(s) \bigr\} \Bigr)ds \Rightarrow \int_0^t Y_k(s)\, ds \ \ \text{in } C[0,\infty).
$$
\end{lemma}


\end{document}